\documentclass[12pt]{article}
\usepackage{amsmath, amsfonts, amsthm,amssymb, color, hyperref, enumerate, extarrows, cite, subfiles, csquotes, enumitem, commath, mathtools,eucal,mathrsfs, bm,url }

 \usepackage[titletoc]{appendix}
\usepackage[T1]{fontenc}

\newtheorem{Lemma}{Lemma}[section]
\newtheorem{Theorem}[Lemma]{Theorem}
\newtheorem{Proposition}[Lemma]{Proposition}
\newtheorem{Corollary}[Lemma]{Corollary}

\newtheorem{remark}[Lemma]{Remark}
\newtheorem{definition}[Lemma]{Definition}
\newtheorem{example}[Lemma]{Example}

\newtheorem{Fact}[Lemma]{Fact}

\hypersetup{
	colorlinks   = true,
	citecolor    = magenta,
    urlcolor    =blue}

\usepackage[titletoc]{appendix}

	\newcommand{\C}{\mathbb{C}}

 \def\bt{\begin{Theorem}}
\def\et{\end{Theorem}}
\def\bl{\begin{Lemma}}
\def\el{\end{Lemma}}
\def\bp{\begin{Proposition}}
\def\ep{\end{Proposition}}
\def\bcor{\begin{Corollary}}
\def\ecor{\end{Corollary}}
\def\bpf{\begin{proof}}
\def\epf{\end{proof}}

\def\brem{\begin{remark}}
\def\erem{\end{remark}}

\def\bedef{\begin{definition}\rm }
\def\endef{\end{definition}}

\def\beg{\begin{example}}
\def\eeg{\end{example}}

\def\bef{\begin{Fact}}
\def\eef{\end{Fact}}

\def\bc{\begin{center}}
\def\ec{\end{center}}

\def\beq{\begin{equation}}
\def\eeq{\end{equation}}
\def\beqarray{\begin{eqnarray*}}
\def\eeqarray{\end{eqnarray*}}
\def\<{\leftangle}
\def\>{\rightangle}
\def\({\left(}
\def\){\right)}

\def\<{\langle}
\def\>{\rangle}

\def\a{\alpha}

\def\t{\tau}

\def\e{\varepsilon}

\def\O{\Omega}

\def\z{\zeta}

\def\w.r.t.{with respect to}
\def\R{{\mathbb{R}}}
\def\N{{\mathbb{N}}}
\def\Z{{\mathbb{Z}}}

\def\C{{\mathbb{C}}}

\def\bq{\begin{quote}}
\def\eq{\end{quote}}

\def\bit{\begin{itemize}}
\def\eit{\end{itemize}}
\def\i{\item}
\def\ben{\begin{enumerate}}
\def\een{\end{enumerate}}

\begin{document}

\title{ \vspace{-1.2cm} \bf On the dimension of the \texorpdfstring{$p$}{p}-Bergman spaces\rm}
\author{Shreedhar Bhat and Achinta Kumar Nandi}
\date{}

\maketitle

\begin{abstract}
The investigation of the dimension of Bergman spaces has long been a central topic in several complex variables, uncovering profound connections with potential theory and function theory since the pioneering work of Carleson, Wiegerinck, and others in the 1960s. We investigate the dimension of $p$-Bergman spaces associated with pseudoconvex domains in $\mathbb{C}^n$. By constructing $L^p$-versions of the extension theorems of Ohsawa and Ohsawa-Takegoshi, we establish several geometric and potential-theoretic criteria that ensure the spaces are infinite-dimensional. Sufficient conditions for the infinite dimensionality of $p$-Bergman spaces of complete N-circled fibered Hartogs domains, balanced domains, and weighted $p$-Fock spaces are obtained by applying the mentioned $L^p$-analogs of extension theorems and generalizing a sufficient condition of Jucha.

\end{abstract}

\renewcommand{\thefootnote}{\fnsymbol{footnote}}
\footnotetext{\hspace*{-7mm} 
\begin{tabular}{@{}r@{}p{16.5cm}@{}}
& Keywords. $p$-Bergman space, Wiegerinck Problem.\\
& Mathematics Subject Classification. Primary 32A36; Secondary 32A70; \end{tabular}}
\section{Introduction} 

For a domain (open connected set)  $\O$ in $\C^n$ and for $p\in [1,\infty)$, the weighted $p$-Bergman space comprises holomorphic functions that are $L^p$ integrable with respect to the weight function $\varphi$ on $\O$,  
$$
A^p(\O,\varphi)=\set{f:\O\rightarrow\C\big\vert \,f\text{ is holomorphic in }\O \,\text{ and }\int_\O \abs{f}^pe^{- p \varphi} <\infty }.
$$
If the weight function is uniform, i.e. $\varphi\equiv 0$, then we just refer to it as $p$-Bergman space, denoted $A^p(\O)$. When $p=2$, this space coincides with the classical Bergman space that possesses a Hilbert structure and a reproducing kernel. For $p\neq2$, $A^p(\O)$ retains rich geometric and functional analytic properties but lacks a Hilbert space structure, making the study of its dimensionality and its basis particularly complex. One of the major themes in Bergman theory is the Wiegerinck problem \cite{wiegerinck1984domains}, which, in its classical form, asks whether the non-triviality of a domain’s Bergman space implies infinite dimensionality—a dichotomy that always holds in one complex variable but is known to fail in higher dimensions. Several works have contributed to the $p=2$ case, such as Wiegerinck’s counterexamples \cite{wiegerinck1984domains}, Gallagher’s explorations of infinite dimensionality in various settings \cite{gallagher2017dimension}, \cite{gallagher2022dimension}, B{\l}ocki, Pflug, Zwonek and their coauthors \cite{blocki2020generalizations}, \cite{pflug2017h}, \cite{WlodzimierzZwonek2000} (also see the references therein). The purpose of this paper is to systematically investigate to what extent the dichotomy of dimensions persists for $p$-Bergman spaces. More specifically, it asks under what conditions on the domain and the integrability parameter $p$, the $p$-Bergman space can contain nontrivial holomorphic functions or be rich enough to have infinitely many linearly independent functions. This problem intertwines theory from functional analysis and potential theory, and remains challenging especially in the setting of non-Hilbert space (i.e. $p\neq 2$).

The present work is devoted to questions of a similar nature concerning the dimension of 
$p$-Bergman spaces on pseudoconvex domains. The dimension of the Bergman space 
has been extensively studied for various classes of pseudoconvex domains. In particular, 
Jucha \cite{jucha2012remark} and Boudreaux \cite{boudreaux2021dimension} investigated 
the Bergman spaces of Hartogs domains with planar and higher-dimensional bases. 
A fundamental result due to Jucha \cite{jucha2012remark} establishes a necessary and 
sufficient condition for the infinite dimensionality of the Bergman space of Hartogs 
domains $D_{\phi}(\mathbb{C})$ with base $\mathbb{C}$, thereby precisely proving the 
\textit{dichotomy of dimensions}. The phenomenon of non-trivial Bergman spaces being infinite dimensional is referred as dichotomy of dimension in the literature (see \cite{jucha2012remark}, \cite{gallagher2017dimension}, \cite{pflug2017h}). Boudreaux, in turn, obtained several sufficient 
conditions for infinite dimensionality by means of extension theorems. Related problems 
have also been studied by Pflug and Zwonek \cite{pflug2017h}, who analyzed balanced 
domains and, among other significant contributions, gave a complete characterization of 
the balanced domains inside $\mathbb{C}^2$ whose Bergman spaces are trivial. We consider questions of a similar spirit for $p$-Bergman spaces.

The dimension of Bergman spaces of planar domains are thoroughly understood. Wiegerinck \cite{wiegerinck1984domains} proved that Bergman spaces of planar domains are either trivial or infinite dimensional. Carleson \cite{Carleson1998SelectedPO} characterized domains with infinite dimensional Bergman spaces in terms of the logarithmic capacity of the complement. These equivalent characterizations of a planar domain $\Omega$ with infinite dimensional Bergman space are summarized (see Theorem $1.4$ of \cite{gallagher2021closed}) as
\begin{enumerate}
    \item  Bergman space $A^2(\Omega)$ is nontrivial,
    \item Bergman space is an infinite dimensional Banach space,
    \item logarithmic capacity of the complement of $\O$ is positive, i.e. $c_0( \Omega^c) >0$,
    \item there exists smooth, strictly subharmonic function $\varphi$ (i.e. $\Delta \varphi > 0$) that is bounded on $\O$.
\end{enumerate}
Several authors obtained sufficient conditions for infinite dimensionality of the Bergman spaces of planar domains (see for instance, Theorem~1.3 in \cite{gallagher2021closed}). The dichotomy of dimension has also been studied for Bergman spaces of holomorphic sections of domains in $\mathbb{P}^1$, and for bundle-valued Bergman spaces of compact Riemann surfaces by Sz\H{o}ke \cite{szHoke2022theorem}, Gallagher-Gupta-Vivas \cite{gallagher2022dimension}, \cite{Gallagher2024}.

Similar questions for $p$-Bergman spaces of planar domains have been considered before, and are perhaps known to experts. To make the article self-contained we study the dimension of $p$-Bergman spaces of planar domains in section $\ref{sec:planar}$. We summarize the sufficient criteria characterizing domains with infinite-dimensional $p$-Bergman space in the following theorem, presented as theorem \ref{th:main.planar}:\\

\textbf{Theorem \ref{th:main.planar}.}
Let $\Omega \subset \mathbb{C}$ be a domain. Then the following statements hold:
\begin{enumerate}
    \item For $p\geq 2$,  $A^p(\Omega)$ is either trivial or it is an infinite-dimensional Banach space.
    \item  \textbf{Carleson\cite{Carleson1998SelectedPO}:} Suppose that  $c_{2-q} (\Omega^c) > 0$  then $A^p(\Omega)$ is non-trivial, where $\frac{1}{p} + \frac{1}{q}=1$ and $p > 2$.
    \item Suppose there exists subharmonic  $\phi: \Omega \to [-\infty,0)$ with $\phi - \e_0\norm{\cdot}^2 \in SH(\Omega)$ for some $\e_0>0$ and $\{z \in \Omega: \nu(\phi,z)=0\}$ contains a non-empty open subset of $\Omega$, where $\nu(\phi,z)$ is the Lelong number of $\phi$ at $z$. Then for $1 \leq p <2$, $A^p(\Omega)$ is infinite dimensional.
\end{enumerate}   
A complete description of $p$-Bergman spaces of planar domains is given next. Non-trivial $p$-Bergman spaces of planar domains are shown to be infinite dimensional when $p \geq 2$. In the remaining case $1 \leq p <2$, planar domains with $k$-dimensional $p$-Bergman spaces are constructed in theorem \ref{th:planardomain}, where is an arbitrary positive integer $k$.
 
The remaining part of the paper concerns $p$-Bergman spaces of pseudoconvex domains in higher dimensions. Apart from certain specific families of pseudoconvex domains little is known about the \textit{dichotomy of dimension} of their Bergman spaces. A reasonable approach to the problem is perhaps by characterizing pseudoconvex domains with infinite dimensional Bergman space. Much effort has been dedicated in this pursuit and several significant geometric, function theoretic sufficient conditions for infinite dimensionality have been obtained (see \cite{gallagher2017dimension}, \cite{gallagher2021closed}, \cite{gallagher2022dimension}, \cite{Gallagher2024}, \cite{boudreaux2021dimension}, \cite{jucha2012remark}, \cite{borichev2021dimension}, etc). A common theme in finding sufficient conditions for infinite dimensionality is to make use of extension theorems to relate the Bergman spaces of co-dimension one hyperplanes, or hypervarities inside the pseudoconvex domains, with the Bergman space of the domain. H\"ormander-Skoda-Bombieri's result, extension theorems in the spirit of Ohsawa-Takegoshi thus become indispensable tools in the pursuit of sufficient conditions for infinite dimensionality. $L^p$-analogs of H\"ormander-Bombieri-Skoda theorem (see proposition \ref{prop:pSkoda}), a version of the Ohsawa-Takegoshi extension theorem due to Dinew \cite{dinew2007ohsawa} (see theorem \ref{th:pOT}), and extension theorems of Ohsawa (\cite[Theorem $1.1$]{ohsawa2017extension},  \cite[Theorem $4.1$ ]{ohsawa2017extensionVIII}) are established in the first part of section $3$ (see theorems \ref{th:pOhsawa1}, and \ref{th:pOhsawa2} respectively). We set up the following existence and extension results with $L^p$-estimates to obtain sufficient conditions for infinite dimensionality of $p$-Bergman spaces. We provide new and elementary proofs of the results above. Our methods are inspired by Chen \cite{Chen2021OnTheory} and Xiong's \cite{xiong2023minimal} idea of considering \textit{minimal} extensions. Our motivation for proving the extension theorems with $L^p$-estimates is to use them to compute $p$-Bergman spaces of unbounded pseudoconvex domains. The consequences of the $L^p$-analogs of extension results enable us to compare the $p$-Bergman spaces of pseudoconvex domains with the $p$-Bergman spaces of hyperplanes and hypersurfaces. See corollaries \ref{cor:pOT}, \ref{cor:pOT2}, \ref{cor:Ohsawa1}. The restriction to $p\in[1,2]$ is motivated by the technical and conceptual necessity: for $p>2$, the tools that underlie many classical results, such as the Ohsawa-Takegoshi extension theorem, Skoda's theorem, etc, become unavailable or lack appropriate analogs. These results are instrumental in constructing nontrivial holomorphic functions with integrability properties, which are especially crucial in proving infinite dimensionality results, and their absence for $p>2$ compels us to limit the scope.

The remaining part of the paper is dedicated to examining $p$-Bergman spaces of special classes of pseudoconvex domains, namely Hartogs domains with one and higher dimensional bases. That is, Hartogs domains with base $G$ are domains of the form
$$
D_{\phi}(G):= \{(z,w) \in \mathbb{C}^M \times \mathbb{C}^N: \|w\| < e^{-\phi(z)} \},
$$
where $G \subset \mathbb{C}^M$ is pseudoconvex, $\phi \in PSH(G)$. Several conditions on the base domain and on the plurisubharmonic weight are obtained that are sufficient for the infinite dimensionality of the $p$-Bergman space of $D_{\phi}(G)$. It is shown that $A^p(D_{\phi}(G))$ can separate points when $\phi - c\|.\|^2 \in PSH(G)$, and its Lelong number $\nu(\phi,.)$ vanishes on some open subset $U$ of $G$ (see theorem \ref{th:point-separating}). The $L^p$-analogs of extension theorems enable us to compare $A^p(D_{\phi}(G))$ with $p$-Bergman spaces $A^p(D_{\phi}(G \cap H))$, where $H$ is a hyperplane or a hypersurface (see corollary \ref{cor:pOTH}, \ref{cor:Ohsawa1}, \ref{cor:Ohsawa2}). Jucha in \cite{jucha2012remark} introduces a necessary and sufficient condition for the infinite dimensionality of Bergman spaces of $D_{\phi}(\mathbb{C})$. Remarkably, he is able to show the failure to meet the condition forces $A^2(D_{\phi}(\mathbb{C}))=\{0\}$, thus proving the dichotomy of dimensions for $A^2(D_{\phi}(\mathbb{C}))$. In contrast to this, we show that failure to meet these conditions do not ensure triviality of $p$-Bergman spaces, when $p\in [1,2)$. The sufficient condition in proposition $3.4$ \cite{jucha2012remark} encapsulates the essence: if the plurisubharmonic weight $\phi$ grows faster than the sum of a plurisubharmonic function $u$ and $M' \log \|z\|$ for an appropriate constant $M'$ near infinity (that is, outside some compact set), and if $u$ has a milder logarithmic leading term singularity than $\phi$ in a compact set, then non-trivial holomorphic functions can be constructed inside $D_{\phi}(\mathbb{C})$. Additional requirements on the Riesz decomposition of the measure $\Delta \phi$ (see condition ($4.2$) in \cite{jucha2012remark}) ensure the existence of non-trivial elements in $A^2(\mathbb{C}, (N+|n|)\phi)$ for infinitely many multi-indices $n \in \mathbb{N}^N$.  We adapt the idea of proposition $3.4$ in \cite{jucha2012remark} and obtain sufficient conditions (\ref{eq:generalJucha}) and (\ref{eq:condnpFock}) that result in the infinite dimensionality of $A^p(D_{\phi}(G))$ and the weighted $p$-Fock space (see theorems \ref{th:generalJucha}, \ref{th:pFock}). As a consequence of theorem \ref{th:generalJucha}, we show that $A^p(D_{\phi}(G))$ is infinite dimensional whenever $\phi - c \|.\|^2 \in PSH(G)$ on some pseudoconvex $G \subset \mathbb{C}^M$ (see corollary \ref{cor:generalJucha}). We also observe that some results of Jucha \cite{jucha2012remark} and Pflug-Zwonek \cite{pflug2017h} do not translate in case of  $p$-Bergman space ($1\leq p<2$). See examples \ref{ex:onedimHartogs}, \ref{ex:elbalanced}.

We study the infinite dimensionality of $p$-Fock-type spaces $\mathcal{F}^p_{\phi}$ in section $\ref{sec:pFock}$. The dimensionality of the weighted Fock space of entire functions is motivated by questions in quantum mechanics, particularly by the study of zero modes, which are eigenfunctions with zero eigenvalues. Haslinger \cite{haslinger2017complex}, Shigekawa \cite{shigekawa1991spectral}, and Borichev-Le-Youssfi \cite{borichev2021dimension} have previously considered dimensionality questions for such spaces. We define weighted $p$-Fock spaces $\mathcal{F}^p_{\phi}$ to be the Banach space of weighted $L^p$-integrable entire functions, with the weights often not required to be plurisubharmonic. Generalizing Jucha's condition $4.2$ in this scenario, condition (\ref{eq:condnpFock}) is obtained and shown to be sufficient for the infinite dimensionality of $\mathcal{F}^p_{\phi}$. An $L^p$-analog of Borichev-Le-Youssfi result \cite[Theorem 3.1]{borichev2021dimension}, and $L^p$-analog of Haslinger's result \cite[Theorem 11.20]{haslinger2017complex}, is proved as a consequence of condition (\ref{eq:condnpFock}), for $1\leq p \leq 2$. Lastly, in section \ref{sec:Reinhardt}, we construct examples of complete Reinhardt domains inside $\mathbb{C}^n$ with finite dimensional $p$-Bergman spaces.

\section{\texorpdfstring{$p$}{p}-Bergman Spaces of Planar Domains}\label{sec:planar}

Suppose $\Omega$ is a non-empty domain inside $\mathbb{C}$. We shall investigate the dimensions of the $p$-Bergman spaces $A^p(\Omega)$, for $1 \leq p < \infty$. Let us start by recalling the notions of logarithmic capacity $c_0(\Omega^c)$ and Riesz $\alpha$-capacity $c_{\a}(\Omega^c)$. Let $K \subset \Omega^c$ be compact, and $\mu$ be a probability measure supported on $K$. The logarithmic energy on $K$ is defined as
$$
I_0(\mu) := \int_K \int_K \log \frac{1}{\norm{z-w}} \,d\mu(z) \,d\mu(w),
$$
and when there is one $\mu$ with finite energy, define
$$
c_0(K) := \exp \left( -\inf_{\mu \in \mathcal{P}(K)} I_0(\mu) \right),
$$
$\mathcal{P}(K)$ being the set of Borel probability measures supported on $K$. If no such finite measure exists, $c_0(K)=0$. We define the logarithmic capacity on $\Omega^c$ as
$$
c_0(\Omega^c):= \sup\{c_0(K): K \subset \Omega^c, K \text{ compact} \}.
$$

The Riesz $\alpha$-capacity for $\alpha \in (0,2)$ is defined in a similar manner. The $\alpha$-energy of a finite positive Borel measure $\mu$ supported on $K$ is defined as
$$
I_{\alpha}(\mu) := \int_K \int_K  \frac{1}{\norm{z-w}^\a} \,d\mu(z) \,d\mu(w),
$$
the $\alpha$-capacity of $K$,
$$
c_{\alpha}(K) := \sup_{\mu \in \mathcal{P}(K)} \frac{1}{I_{\alpha}(\mu)},
$$
and $\a$-capacity of $\O^c$ is given by  
$$
c_{\alpha}(\Omega^c):= \sup\{c_{\alpha}(K): K \subset \Omega^c, K \text{ compact} \}.
$$
Following is the main theorem of this section:

\bt \label{th:main.planar}
Let $\Omega \subset \mathbb{C}$ be a domain. Then the following statements hold:
\begin{enumerate}
    \item For $p\geq 2$,  $A^p(\Omega)$ is either trivial or it is an infinite-dimensional Banach space.
    \item  \textbf{Carleson\cite{Carleson1998SelectedPO}:} Suppose that  $c_{2-q} (\Omega^c) > 0$  then $A^p(\Omega)$ is non-trivial, where $\frac{1}{p} + \frac{1}{q}=1$ and $p > 2$.
    \item Suppose there exists subharmonic  $\phi: \Omega \to [-\infty,0)$ with $\phi - \e_0\norm{\cdot}^2 \in SH(\Omega)$ for some $\e_0>0$ and $\{z \in \Omega: \nu(\phi,z)=0\}$ contains a non-empty open subset of $\Omega$, then for $1 \leq p <2$, $A^p(\Omega)$ is infinite dimensional.
\end{enumerate}
\et
The proof will be a combination of two results proved later in this section. Note that, the third result in the theorem above is almost relating to non-triviality of $A^2(\Omega)$ with the infinite dimensionality of $A^p(\Omega)$. For when the Bergman space of $\Omega$ is non-trivial, using Carleson's result it can be shown that there exists a smooth strictly subharmonic function $\phi$ on $\Omega$, and point-separating functions can be constructed as an application of H\"ormander's theorem. 

Relating to Wiegerinck's work, a natural question to investigate is whether there exists planar domains with finite dimensional $p$-Bergman spaces. Restricting to the case $1\leq p <2$ it is shown that there are planar domains $\Omega_k$ with $k$-dimensional $p$-Bergman spaces. In contrast, for $p \geq 2$, all non-trivial $p$-Bergman spaces of planar domains must be infinite dimensional. Although this result is known, we include the proofs for sake of completeness.
\begin{Theorem}\label{th:planardomain}
The dimensionality of $p$-Bergman space of planar domains exhibits the following behavior :  
\begin{itemize}
    \item For $1 \leq p < 2$ and for each $k \in \mathbb{N}$, there exists a domain $\Omega_k \subset \mathbb{C}$ such that $\dim A^p (\Omega_k) = k$.
    \item For $ p\geq 2$ and $\Omega \subset \mathbb{C}$, $A^p(\Omega)$ is either trivial or it is an infinite dimensional Banach space.
\end{itemize}
\end{Theorem}

We will prove this theorem in the next few propositions. Note that $A^p(\mathbb{C} \setminus \{a\})$ is trivial for $a \in \mathbb{C}$, whenever $1 \leq p < \infty$. Define $\Omega_k := \mathbb{C} \setminus A_{k+1}$ where $A_{k+1}$ is a finite subset containing $k+1$ elements with $k \geq 1$.
\begin{Proposition}
    For $p\in [1,2) $, $A^p(\Omega_k)$ is a $k$-dimensional Banach space.
\end{Proposition}

\begin{proof}
Assume that $A_{k+1}:= \{a_0, a_1, \dots, a_{k}\}$, where $a_j \in \mathbb{C}$, is a finite set with $k+1$ elements. Define
$$
f_j(z) = \frac{1}{z-a_0} - \frac{1}{z-a_j}\, ;\hspace{3mm} 1 \leq j \leq k.
$$
First, we shall show that $f_j \in A^p(\Omega_k)$ for all $1 \leq j \leq k$.

Suppose $r > 0$ be such that $|a_i - a_j| > 2r$ for all $i,j$. Then it is enough to establish that $f_j \in A^p(B(z_1, r)\setminus \set{z_j})$ or equivalently $f_j \in A^p(B(z_j,r))$. Consider
\begin{align*}
    \int_{B(a_0,r)} |f_j|^p \,d\lambda^2 &= \int_{B(a_0,r)} \frac{|a_0-a_j|^p}{|(z-a_0)(z-a_j)|^p} \,d\lambda^2 \\
    &\leq \frac{|a_0 - a_j|^p}{r^p} \int_{B(a_0,r)} \frac{1}{|z-a_0|^p} \,d\lambda^2 \\
    &= 2 \pi M_j \left[ \frac{s^{2-p}}{2-p} \right]_0^r = \frac{2 \pi r^{2-p} M_j}{2-p} < \infty,
\end{align*}
where $M_j = \frac{|a_0 - a_j|^p}{r^p}$. Since the $\displaystyle{\set{f_j}_{j=1}^k}$ is linearly independent, the above computation illustrates $\dim A^p(\Omega_k) \geq~k$. 

Lastly to see that $A^p(\O_k)$ is a $k$ dimensional space, consider $f \in A^p(\Omega_k)$, and $\displaystyle{f(z) = \sum\limits_{-\infty}^{\infty} b^i_j (z-a_i)^j}$
be the Laurent expansion of $f$ in a neighbourhood of $a_i$. Note that, for $j<0$,
\begin{align*}
    b^i_{j}&=\dfrac{1}{2\pi }\int_{0}^{2\pi} f(a_i+re^{it}) (re^{it})^{-j} dt,\\
    {b^i_{j}} \int_{0}^\e rdr&=\dfrac{1}{2\pi } \int_0^\e \int_{0}^{2\pi}f(a_i+re^{it}) (re^{it})^{-j} rdr dt, \\
    \abs{b^i_{j}} \e^2 & \leq  \norm{f}_{A^p(B(a_i, \e))} \cdot \left[\int_{0}^\e r^{(-qj+1)} dr \right]^{1/q}\leq  \norm{f}_{A^p(\O)} \left[\dfrac{\e^{(-qj+2)}}{-qj+2} \right]^{1/q},\\
    \abs{b^i_{j}} &\leq C\norm{ f}_{A^p(\O)} \e^{-j-2+2/q}.
\end{align*}
The inequality is true for all $\e>0$. Therefore setting $\e\rightarrow 0$, $j\leq -2 \implies $ $\abs{b^i_{j}}=0$. \\
Therefore $f$ has at most simple poles at $\displaystyle{\{a_i\}_{i=0}^{k}}$, i.e. $f$ can be written as  $$f(\z)=\dfrac{g(\z)}{\prod\limits_{i=i_1}^{i_l}(z-a_i)}\, ,$$ where $g$ is an entire function and $a_{i_1},a_{i_2}, \cdots, a_{i_l}$ are the poles of $f$, and $2\leq l \leq k$.\\
Further, we claim that $g$ is a polynomial of degree at most $l-2$.
To prove this, observe that  
\begin{align*}\dfrac{d^jg}{dz^j} (0) & =\dfrac{1}{2\pi j!}\int_{0}^{2\pi} g(re^{it}) (re^{it})^{-j} dt,\\
\abs{\dfrac{d^j g}{dz^j}(0)} \int_0^R rdr &\simeq  \abs{ \int_0^R \int_0^{2\pi } \dfrac{g(re^{it})}{(re^{it}-a_{i_1})(re^{it}-a_{i_2})\cdots (re^{it}-a_{i_l}) } r^l (re^{it})^{-j} rdr dt },\\
\abs{\dfrac{d^jg}{dz^j}(0)}&\leq C\dfrac{2}{R^2} \norm{f}_{A^p(\O_k)} \left[\int_{0}^R r^{lq-jq+1}dr\right]^{1/q}\simeq \norm{f}_{A^p(\O_k)} R^{l-j+2/q-2}.
\end{align*} 
The inequality is true for all $R>0$. Therefore setting $R\rightarrow \infty$, $j\geq l-1 \implies $ $\dfrac{d^j g}{dz^j}=0$ i.e. $g$ is a polynomial of degree at most $l-2$. Now, using partial fractions, we obtain $\set{B_j}$ uniquely such that $$f=\sum_{j=1}^l B_j \left[ \dfrac{1}{z-a_0}-\dfrac{1}{z-a_{i_j}}\right].$$\\
\end{proof}

\brem 
The above result says that for a planar domain $\O$ and $p\in [1,2)$,  $A^p(\O)$ is nontrivial if and only if $\O$ is hyperbolic.
\erem

We now approach the $p$-analog of Wiegerinck's result, for $p \geq 2$. That is, we establish that non-trivial $p$-Bergman spaces of planar domains are in fact infinite dimensional, for $p \geq 2$.

Given any planar domain $\Omega$ containing the origin, the linear fractional transformation $\displaystyle{\Psi(z) = \frac{1}{z}}$ maps $\Omega$ to a neighborhood of $\infty$ in the Riemann sphere. We notice that $A^p(\Omega) \cong A^p(\Psi(\Omega), \log \abs{z}^4)$. Indeed, for any $g \in A^p(\Omega)$,
$$
\int_{\Omega} |g(z)|^p \,d\lambda^2 = \int_{\Psi(\Omega)} |g(\Psi^{-1}(\z))|^p |\z|^{-4} \,d\lambda^2
$$
making the assignment $g \mapsto g\circ\Psi^{-1}$ an invertible linear map from $A^p(\Omega)$ to the weighted space $A^p(\Omega, \log |z|^4)$. Hence for our purposes, it is enough to show that every non-trivial weighted $p$-Bergman space $A^p(\Psi(\Omega), \log |z|^{4})$ is infinite-dimensional to settle the case when the planar domain $\Omega$ does not contain a neighborhood of {infinity}. The remaining case requires a similar argument which we consider afterwards.

\bp
Every non-trivial weighted $p$-Bergman space $A^p(\Omega, \log \abs{z}^4)$ is infinite-dimensional whenever $p \geq 2$, and $\Omega$ contains a neighborhood of infinity in the Riemann sphere.
\ep

\begin{proof}
Firstly, we consider the case when $A^p(\Omega, \log|z|^{4})$ contains a non-trivial rational function $g$. Since the integral
$$
\int_{\mathbb{C}} |g(z)|^p |z|^{-4} \,d\lambda^2 
$$
diverges, the complement $\Omega^c$ has positive two-dimensional Lebesgue measure. Note that there exists $R > 0$ large enough such that $\Omega^c$ is properly contained in the ball $B(0, R)$.

Considering the Cauchy transform on $\Omega^c$
$$
h(z) = \int_{\Omega^c} \frac{1}{\zeta - z} \,d\lambda^2,
$$
for all $z \in \Omega$, we obtain a non-trivial bounded holomorphic function on $\Omega$. Note that by construction $0 \in \Omega^c$, and there exists $R > 0$ large enough such that $\partial B(0,R) \subset \Omega$. 
We may decompose the integral on the set $\Omega$ as follows
$$
\int_{\Omega} |h(z)|^p |z|^{-4} \,d\lambda^2 = \int_{\Omega \cap B(0,R)} |h(z)|^p |z|^{-4} \,d\lambda^2 + \int_{\Omega \cap B(\infty, R)} |h(z)|^p |z|^{-4} \,d\lambda^2,
$$
where $B(\infty,R)=\{\abs{w}>R\}$. Using the boundedness and continuity of $h$ and $\frac{1}{z}$ on $\Omega \cap B(0.R)$ the first integral on the right-hand side is seen to be convergent. Also, since both $h(z) \to 0$ and $\frac{1}{z} \to 0$ as $|z| \to \infty$, the next integral is also finite. Thus, the Cauchy transform $h$ lies in the weighted $p$-Bergman space $A^p(\Omega, \log \abs{z}^4)$. Using the same calculation as above, the functions $h^2, h^3, \dots, h^n, \dots$ lie in the space $A^p(\Omega, \log \abs{z}^4)$. This completes the proof when $g$ is rational.

Now suppose $g \in A^p(\Omega, \log \abs{z}^4)$ is non-rational. So, we can assume that the Laurent expansion of $g$ around infinity is
$$
g(z) = \sum\limits_{k=\ell}^{\infty} a_k z^{-k},
$$
where $\ell \geq 1$ whenever $p \geq 2$, and $a_{\ell} \neq 0$.

We construct a non-rational function $h$ inside $A^p(\Omega, \log \abs{z}^4)$ whose Laurent expansion about infinity does not contain the terms $z^{-1}, z^{-2}, \dots, z^{-\ell}$ as follows. Choose points $z_1, \dots, z_{\ell + 1}$ from $\Omega$, and define
\begin{equation}\label{hnon-r}
    h(z) = \sum_{1}^{\ell + 1} b_j \left( \frac{g(z) - g(z_j)}{z - z_j} \right), 
\end{equation}
where the coefficients $b_j$'s are to be determined. Also, $h$ can be written near infinity in the form 
$$
h(z)= \sum\limits_{k=1}^{\infty} c_k z^{-k}.
$$
Choosing $R > 0$ large enough such that $B(\infty, R) \subset \{z \in \Omega: |z_j| < |z|, 1 \leq j \leq p+1\}$, we compare the Laurent expansion of $h$ with its definition in (\ref{hnon-r})
$$
h(z) = \sum_{1}^{\ell + 1} b_j \left( -\sum_{k=1}^{\ell} \frac{g(z_j) z^{k-1}_j}{z^{k}} + \sum_{r=1}^{\infty} \frac{\Sigma_{k=0}^{r-1} a_{\ell + k} z^{r-1 -k}_j - g(z_j) z^{\ell + r -1}_j}{z^{\ell + r}} \right),
$$
and obtain
$$
B_k = -\sum_{1}^{\ell + 1} b_j g(z_j) z^{k-1}_j,
$$
for all $k = 1, \dots, \ell$. The homogeneous system of linear equations
$$
\sum_1^{\ell + 1} b_j g(z_j) z^{k-1}_j = 0,
$$
has solutions $b_j \in \mathbb{C}$. Thus, there exists $b_j$ such that
$$
B_1 = \dots = B_{\ell} = 0.
$$

Hence, by construction $h$ is a non-rational holomorphic function in $A^p(\Omega, \log \abs{z}^4)$ that does not have the terms $z^{-1}, \dots, z^{-\ell}$ in its Laurent expansion. In particular, it is different from the chosen $g$. Applying the above algorithm on $h$, a function $h_1 \in A^p(\Omega, \log \abs{z}^4)$ can be obtained such that the expansion of $h_1$ does not contain $z^{-1}, \dots, z^{-(\ell + 1)}$. Thus $A^p(\Omega, \log \abs{z}^4)$ is an infinite dimensional Banach space.
\end{proof}

The remaining case is when $\Omega$ contains a neighborhood of infinity. In this situation, it can be shown that every non-trivial $A^p(\Omega)$ is infinite-dimensional. Precisely, the only adaptation of the above argument required to settle this case is that for non-rational $g \in A^p(\Omega)$, the expansion of $g$ about infinity
$$
g(z) = \sum_{k=\ell}^{\infty} a_k z^{-k},
$$
must have $\ell \geq 2$ when $p = 2$, and $\ell \geq 1$ whenever $p>2$. Afterward, a different non-rational $h$ can be constructed in $A^p(\Omega)$ in an identical fashion, establishing the infinite dimensionality of $A^p(\Omega)$.
This completes the proof of the equivalence $(1)$ in theorem \ref{th:main.planar}, the $p$-dimensional analog of Wiegerinck's result for planar domains whenever $p \geq 2$.\\
The following proposition completes the proof of theorem \ref{th:main.planar}.
\begin{Proposition}
Suppose there exists subharmonic  $\phi: \Omega \to [-\infty,0)$ with $\phi - \e_0\norm{\cdot}^2 \in SH(\Omega)$ for some $\e_0>0$ and $\{z \in \Omega: \nu(\phi,z)=0\}$ contains a non-empty open subset of $\Omega$, then for $1 \leq p <2$, $A^p(\Omega)$ is infinite dimensional.
\end{Proposition}

Proof is identical to theorem \ref{th:point-sep}. Note that, by theorem $1.3$ of \cite{gallagher2021closed} the infinite dimensionality of $A^p(\Omega)$ can be equivalently seen as consequences of the $\bar \partial$-operator having closed range on $L^2(\Omega)$, or, the existence of a Poincar\'e-Dirichlet inequality on $\Omega$, or, $\Omega$ having a finite \textit{capacity inradius} $\rho_{cap}(\Omega) < \infty$.

\section{Infinite Dimensionality of \texorpdfstring{$p$}{p}-Bergman Spaces}

\subsection{\texorpdfstring{$L^p$}{Lp}-tools to study dimensionality}\label{subsec:tools}
In this subsection, we obtain certain sufficient conditions on pseudoconvex domains that ensure the infinite dimensionality of their $p$-Bergman spaces. We prove $L^p$-analogs of some classical extension theorems and as applications obtain sufficient conditions for infinite dimensionality of $p$-Bergman spaces of pseudoconvex domains.

For the sake of completeness, we recall a few well-known notions concerning the singularities of subharmonic functions. The \textbf{Lelong number} of $\phi \in SH(G)$ at $z$ is denoted by $\nu(\phi,z)$ and is defined by
$$
\nu(\phi,z) := \liminf_{w \to z} \frac{\phi(w)}{\log |w-z|}.
$$
By \cite[Proposition 2.2]{jucha2012remark}, when $\phi$ is a subharmonic function on a planar domain, it can also be described as $$\nu(\phi,z)=\inf \set{t>0: e^{-2\phi/t} \text{ is integrable in some neighbourhood of }z }.$$
We denote by $\lfloor x \rfloor$ the integer part of $x \in \mathbb{R}$, by $\text{frac}(x)= x - \lfloor x \rfloor$ the fractional part of $x \in \mathbb{R}$.

First, we make an observation regarding the $L^p$-analog of H\"ormander-Bombieri-Skoda theorem (see \cite{bombieri1970algebraic} and also \cite{boudreaux2021dimension}).
\begin{Proposition}\label{prop:pSkoda}
Let $1\leq p <2$. Let $u$ be a plurisubharmonic function on a pseudoconvex domain $G \subset \mathbb{C}^M$. If $e^{-2u}$ is integrable in the neighborhood of a point $z_0 \in G$, then for any $\epsilon >0$ there exists $f \in \mathcal{O}(G)$ with $f(z_0)=1$ and
\begin{equation}\label{eq:pSkoda}
\begin{aligned}
\int_G \frac{|f(z)|^p e^{-pu(z)}}{(1+\norm{z}^2)^{M+\epsilon}} \,d\lambda^{2M}(z) &< \infty.
\end{aligned}
\end{equation}
\end{Proposition}

\begin{proof}
We shall only apply H\"older's inequality and use the $L^2$ H\"ormander-Bombieri-Skoda theorem. Using H\"older's inequality, we obtain
\begin{align*}
    \int_G \frac{|f(z)|^p e^{-pu}}{(1+\norm{z}^2)^{M+\epsilon}} \,d \lambda^{2M} &\leq \left( \int_G \frac{|f(z)|^2 e^{-2u}}{(1+\norm{z}^2)^{M+\epsilon}} \,d \lambda^{2M} \right)^{p/2} \left( \int_G \frac{1}{(1+\norm{z}^2)^{M+\epsilon}} \,d \lambda^{2M} \right)^{\frac{2-p}{2}}.
\end{align*}
The $L^2$ H\"ormander-Bombieri-Skoda theorem asserts the existence of such a non-trivial function $f$. Since the other integral is bounded on $\mathbb{C}^M$, we are done.
\end{proof}

\begin{remark}\label{pSkoda:entire}
In particular, applying H\"older inequality as above we can obtain $L^p$-integrable non-trivial entire function on $\mathbb{C}^n$ analogous to Bombieri (see \cite[section IV]{bombieri1970algebraic}).  
\end{remark}
 
More can be said if restrictions on the plurisubharmonic weights are enforced. We show that the $p$-Bergman space of a pseudoconvex domain $G$ that admits a negative-valued plurisubharmonic function $\phi$ with $\phi -c\|\cdot\|^2 \in PSH(G)$, for some $c>0$, and if moreover the Lelong number of $\phi$ is zero on an open $U \subset G$, then $A^p(G)$ is infinite-dimensional. In particular, for such a plurisubharmonic weight the $p$-Bergman space $A^p(G)$ separates points in $U$.

\begin{Theorem}\label{th:point-sep}
Let $1 \leq p <2$. Let $G \subset \mathbb{C}^M$ be a pseudoconvex domain, $\phi \in PSH(G)$ with $\phi(G) \subset [-\infty,0)$ such that $\phi - c\|\cdot\|^2 \in PSH(G)$, for some $c>0$. Moreover, assume $U$ is an open subset of $G$ such that the Lelong number $\nu(\phi,\cdot)=0$ on $U$. Then $A^p(G)$ is infinite-dimensional.     
\end{Theorem}

\begin{proof}
Consider $c_1 >0$ small enough such that $c\|z\|^2 - c_1 (M+\epsilon) \log (1+\|z\|^2)$ is plurisubharmonic in $G$, and define
$$
\hat{\phi} := \phi - c_1 (M+\epsilon) \log (1+\|\cdot\|^2).
$$
Also, the Lelong number of $\hat{\phi}$ is same as $\phi$ on $G$, and thus is zero on $U$.

Now imitating lemma $6$ in \cite{gallagher2017dimension}, for finitely many points $p_1, \dots, p_k \in U$ and $ \e<\frac{1}{2} \norm{p_i-p_j}$, we set the $\bar \partial$-data to be $v= \bar \partial \chi(z-p_k)$, where $\chi:\mathbb{C}^n \to [0,1]$ is a bump function that is zero on $B(0,\frac{\epsilon}{3})$ and one on $\mathbb{C}^n \setminus B(0, \frac{2\epsilon}{3})$. Setting
\begin{align*}
\Phi(z)&:= K \hat{\phi}(z) + \sum_1^k M \chi (z-p_j) \log \norm{z-p_j}, \\
\Phi_1 (z)&:= K \phi(z) + \sum_1^k M \chi (z-p_j) \log \norm{z-p_j},
\end{align*}
choose large $K>0$ such that $\Phi \in PSH(G)$, $\Phi - \|\cdot\|^2 \in PSH(B(p_k,\epsilon))$, and $Kc_1 >1$. Note that the hypothesis on the Lelong number ensures $e^{-2K\phi} \in L^1_{loc}(U)$, for any $K >0$. Solving an appropriate $\bar \partial$-problem one finds $u$ such that
$$
\int_{G} |u|^2 e^{-2 \Phi} \,d \lambda^{2M} < \infty,
$$
forcing $u(p_j)=0$ for all $1 \leq j \leq k$.
Applying reverse H\"older inequality with $\frac{2}{p} \in (1, \infty)$ it can be seen that
\begin{equation}\label{eq:point-separatingRH\"older}
\begin{aligned}
\left( \int_G |u|^p e^{-p\Phi - pKc_1(M+\epsilon) \log (1+\|z\|^2)} \,d \lambda^{2M}\right)^{\frac{2}{p}} &\left( \int_G e^{-\frac{2p}{2-p}K c_1(M + \epsilon) \log (1+\|z\|^2)} \,d \lambda^{2M} \right)^{\frac{p-2}{p}} \\
&\leq \int_{G} |u|^2 e^{-2 \Phi} \,d \lambda^{2M} < \infty.
\end{aligned}
\end{equation}
As the second integral on the left hand side of equation (\ref{eq:point-separatingRH\"older}) is non-zero and finite, we obtain
$$
\int_G |u|^p e^{-p\Phi_1} \,d \lambda^{2M} < \infty,
$$
$u \in L^p(G)$ as $\phi$ is negative. Hence, $\displaystyle{f_k(z):= \chi(z-p_k) - u(z)} \in A^p(G)$ such that $f_K(p_j)=0$ for $j \neq k$ and $f_k(p_k)=1$. Since $k$ was arbitrary and $\{f_k\}$ are linearly independent, $A^p(G)$ is an infinite dimensional Banach space.
\end{proof}

We now turn to classical theorems on extensions of holomorphic functions. These results enable us to compare the $p$-Bergman spaces of a lower dimensional hyperplane or hypervariety $H$ inside a pseudoconvex domain $G \subset \mathbb{C}^M$. In the following, we prove $L^p$-analogs of these results on unbounded pseudoconvex domains. Inspired by Chen \cite{Chen2021OnTheory} and Xiong \cite{xiong2023minimal}, our arguments obtain non-trivial minimal extensions, and outline a general algorithm to obtain extension results with analogous $L^p$-estimates ($1 \leq p < 2$) when the $L^2$-estimates are known.

Consider the following version of the Ohsawa-Takegoshi extension theorem due to Dinew.
\begin{Theorem}\cite{dinew2007ohsawa} \label{th:OTL^2}
Let $D \subset \Omega \times \mathbb{C}^{n-1}$ be pseudoconvex, $\Omega$ be a planar domain, $\phi \in PSH(D)$, $0 \in D$, $H = \{z_1=0\}$, and let $f$ be a holomorphic function on $D \cap H$ satisfying
$$
\int_{D \cap H} |f(z)|^2 e^{-\phi(z)} \,d\lambda^{2n-2} < \infty.
$$
Then there exists a holomorphic function $F$ on $D$ with $F|_{D \cap H} \equiv f$ such that
$$
\int_D |F(z)|^2 e^{-\phi(z)} \,d\lambda^{2n} \leq \frac{4 \pi}{(c(\Omega,0))^2} \int_{D \cap H} |f(z')|^2 e^{-\phi(z)} \,d \lambda^{2n-2}.
$$
\end{Theorem}

Here $c(\Omega,z)= \exp (\lim_{\zeta \to z} G_{\Omega}(z, \zeta) - \log \|z-\zeta\|)$ is the \textit{logarithmic capacity} of $\O$, and $G_{\Omega}(z,\cdot)$ is the negative Green's function of a domain $\Omega \subset \mathbb{C}$ with a pole at $z$. Let us denote the weighted $p$-Bergman space with weight $\phi$ by $A^p(D, \phi)$.

We establish an $L^p$-analog of the above version of the Ohsawa-Takegoshi extension theorem.
\begin{Theorem}\label{th:pOT}
Let $1 \leq p < 2$, $\phi, D, \Omega, H$ be as in Theorem \ref{th:OTL^2}, and let $f$ be holomorphic in $D \cap H$ with
$$
\int_{D \cap H} |f(z)|^p e^{-\phi(z)} \,d\lambda^{2n-2} < \infty.
$$
Then there exists holomorphic $F$ on $D$ with $F|_{D \cap H} \equiv f$ such that
$$
\int_D |F(z)|^p e^{-\phi(z)} \,d\lambda^{2n} \leq \frac{4 \pi}{(c(\Omega,0))^2} \int_{D \cap H} |f(z')|^p e^{-\phi(z)} \,d \lambda^{2n-2}.
$$
\end{Theorem}

\begin{proof}
Without loss of generality, assume $f(0)=1$. Consider a standard smooth plurisubharmonic regularization of $\phi$. That is, $\{D_k: k \in \mathbb{N}\}$ is a family of bounded pseudoconvex domains inside $D$ such that
\begin{itemize}
    \item $\overline{D}_{k-1} \subset D_k$,
    \item $0 \in D_k$, for all $k \in \mathbb{N}$,
    \item $\cup_k D_k = D$,
\end{itemize}
and $\phi_k \in \mathcal{C}^{\infty}(\overline{D}_k) \cap PSH(D_k)$ with $\phi_k \searrow \phi$ uniformly on compact subsets.

Define 
$$
E^p(D_k, \phi_k):= \{F \in \mathcal{O}(D_k): \int_{D_k} |F|^p e^{-\phi_k} \,d \lambda^{2M} < \infty, F|_{D\cap H} \equiv f\},
$$
and using the smoothness of $\phi_k$ and boundedness of $D_k$, every holomorphic extension of $f$ in $D_k$ obtained using Oka-Cartan theory, lies in $E^p(D_k, \phi_k)$. Using $p$-Bergman inequality of \cite{Chen2021OnTheory}, $E^p(D_k,\phi_k)$ is a normal family and thus admits a minimal element with respect to the norm $\displaystyle{\|\cdot\|_{p,\phi_k} = \int_{D_k} |\cdot|^p e^{-\phi_k} \,d \lambda^{2M}}$ on $D_k$. We denote it by $F_k \in E^p(D_k,\phi_k)$.

Define, 
\begin{align*}
E^2(D_k, \phi_k + (2-p) \log |F_k|):= \{F \in \mathcal{O}(D_k): &\int_{D_k} |F|^2 e^{-\phi_k - (2-p) \log |F_k|} \,d \lambda^{2n} \\
&\leq \frac{4 \pi}{c(\Omega,0)^2} \int_{D_k \cap H} |f(z')|^p e^{-\phi_k} \,d \lambda^{2(n-1)} \}.
\end{align*}
Using the Ohsawa-Takegoshi extension theorem, $E^2(D_k, \phi_k + (2-p) \log |F_k|)$ must be non-empty (see theorem $3$ in \cite{dinew2007ohsawa}). Pick, $g \in E^2(D_k, \phi_k + (2-p) \log |F_k|)$. Now due to the reverse H\"older's inequality with exponent $\frac{2}{p} \in (1, \infty)$,
\begin{align*}
\left(\int_{D_k} |g|^p e^{-\phi_k} \,d \lambda^{2n} \right)^{\frac{2}{p}} \left( \int_{D_k} |F_k|^p e^{-\phi_k} \,d \lambda^{2n} \right)^{-(\frac{2}{p}-1)}
&\leq \int_{D_k} |g|^2 |F_k|^{p-2} e^{-\phi_k} \,d \lambda^{2n} \hspace{1.5mm}\text{and since} \\
\int_{D_k} |g|^2 |F_k|^{p-2} e^{-\phi_k} \,d \lambda^{2n} &\leq \frac{4\pi}{c(\Omega,0)^2} \int_{D_k \cap H} |f(z')|^p e^{-\phi_k} \,d \lambda^{2(n-1)},
\end{align*}
we have shown that $g \in E^p(D_k, \phi_k)$. The minimality of $F_k$ in $E^p(D_k, \phi_k)$ forces $\|F_k\|_{p, \phi_k} \leq \|g\|_{p,\phi_k}$. Thus,
$$
\int_{D_k} |F_k|^p e^{-\phi_k} \,d \lambda^{2n} \leq \frac{4\pi}{c(\Omega,0)^2} \int_{D_k \cap H} |f(z')|^p e^{-\phi_k} \,d \lambda^{2(n-1)}.
$$

Now we shall employ a standard diagonal sequence argument as follows:\\
Fix $k\in \N$ and consider a sequence $(F_{\ell})_{\ell \geq k} \subset E^p(D_k, \phi_k)$ of minimal elements of $E^p(D_{\ell}, \phi_{\ell})$. To argue $F_{\ell} \in E^p(D_k, \phi_k)$, note
$$
\int_{D_k} |F_{\ell}|^p e^{-\phi_k} \,d \lambda^{2n} \leq \int_{D_{\ell}} |F_{\ell}|^p e^{-\phi_{\ell}} \,d \lambda^{2n} \lesssim \frac{4\pi}{c(\Omega,0)^2} \int_{D \cap H} |f(z')|^p e^{-\phi} \,d \lambda^{2(n-1)},
$$
for all $\ell \geq k$. As the bound on the RHS is independent of $\ell$, it establishes that $(F_{\ell})_{\ell \geq k}$ constitutes a normal family, and thus has a convergent subsequence which we denote by $(F_{k,\ell})_{\ell \in \mathbb{N}} \subset E^p(D_k, \phi_k)$. Next, we consider this subsequence inside $E^p(D_{k+1}, \phi_{k+1})$, and this is again a normal family admitting a subsequence $(F_{k+1, \ell})_{\ell \in \mathbb{N}}$ convergent inside $E^p(D_{k+1}, \phi_{k+1})$. Continuing in this manner, we obtain the diagonal subsequence $(F^{\infty}_{\ell})_{\ell \in \mathbb{N}}$; a subsequence of each $(F_{k+r, \ell})_{\ell \in \mathbb{N}}$ convergent inside $E^p(D_{k+r}, \phi_{k+r})$, for all $r \in \mathbb{N} \cup \{0\}$. Note that, each $F^{\infty}_{\ell}$ is a minimal element of some $E^p(D_r, \phi_r)$. Suppose $F^{\infty}_{\ell}$ is a minimal element of $E^p(D_{\ell+r_{\ell}}, \phi_{\ell + r_{\ell}})$. Define,
\begin{align*}
    \hat{F}^{\infty}_{\ell}(z) &\equiv F^{\infty}_{\ell}(z),\hspace{1.5mm} z \in D_{\ell + r_{\ell}}, \\
    &\equiv 0,\hspace{1.5mm}\text{otherwise}.
\end{align*}
Now $(\hat{F}^{\infty}_{\ell})_{\ell \in \mathbb{N}}$ converges uniformly on compact sets to some $F^{\infty}$, and $F^{\infty} \in \mathcal{O}(D)$. Applying Dominated Convergence Theorem (DCT), we see
\begin{align*}
\int_{D} |F^{\infty}|^p e^{-\phi} \,d \lambda^{2n} &= \lim_{\ell \to \infty} \int_D |\hat{F}^{\infty}_{\ell}|^p e^{-\phi_{\ell + r_{\ell}}} \,d \lambda^{2n} \\
&= \lim_{\ell \to \infty} \int_{D_{\ell + r_{\ell}}} |F^{\infty}_{\ell}|^p e^{-\phi_{\ell+r_{\ell}}} \,d \lambda^{2n} \\
&\leq \lim_{\ell \to \infty} \frac{4\pi}{c(\Omega,0)^2} \int_{D_{\ell + r_{\ell}} \cap H} |f(z')|^p e^{-\phi_{\ell+r_{\ell}}} \,d \lambda^{2(n-1)} \\
&\leq \frac{4\pi}{c(\Omega,0)^2} \int_{D \cap H} |f(z')|^p e^{-\phi} \,d \lambda^{2(n-1)}.
\end{align*}
This completes the proof.
\end{proof}
We denote the complex linear span of a vector $V \in \mathbb{C}^n$ by $\mathcal{L}(V)=\{\lambda V: \lambda \in \mathbb{C}\}$. As an immediate consequence of theorem \ref{th:pOT} above, we can determine the $p$-Bergman space of a pseudoconvex domain $\Omega \subset \mathbb{C}^2$ as follows.
\begin{Corollary}\label{cor:pOT}
Suppose $1\leq p < 2$. Let $G \subset \mathbb{C}^2$ be pseudoconvex and $X \in \mathbb{C}^2$ be such that the image $\pi(G) \subset \mathcal{L}(X)$ of the orthogonal projection $\pi: \mathbb{C}^2 \to \mathcal{L}(X)$ has a non-polar complement in $\mathcal{L}(X)$. Then $\dim A^p(G)$ is at least $k$ if there exists a section $L$ orthogonal to $\mathcal{L}(X)$ such that cardinality of $\mathbb{C} \setminus L$ is $k+1$. In particular, if $L$ has an infinite or non-polar complement then $\dim A^p(G)$ is infinite.
\end{Corollary}

\begin{proof}
Since $\pi(G) \subset \mathcal{L}(X)$ has non-polar complement, we think of $G \subset \pi(G) \times \mathbb{C}$ and apply theorem \ref{th:pOT}. Thus, every non-trivial $f \in A^p(L)$ can be extended to a non-trivial $F \in A^p(G)$. By theorem \ref{th:planardomain}, dimension of $A^p(L) = k$ if $\mathbb{C} \setminus L$ contains $k+1$ elements. 
\end{proof}

Weighted $p$-Bergman space of a pseudoconvex domain $\Omega \subset \mathbb{C}^n$ can be similarly compared with the weighted $p$-Bergman space of $H \cap \Omega$, where $H$ is a complex hyperplane.
\begin{Corollary}\label{cor:pOT2}
Suppose $1\leq p <2$, $\Omega \subset \mathbb{C}^n (n \geq 2)$ be a pseudoconvex domain containing the origin, $\phi \in PSH(\Omega)$, and $T: \mathbb{C}^n \to \mathbb{C}$ be a non-zero linear map such that $T(\O)$ has a non-polar complement. Then
$$
\dim A^p(\ker T \cap \Omega, \phi) \leq \dim A^p(\Omega, \phi).
$$
\end{Corollary}

\begin{proof}
Making a linear change of coordinates, consider $T$ to be the projection $z \mapsto z_1$. Although $p$-Bergman spaces are not preserved by biholomorphisms, they are isomorphic to some weighted $p$-Bergman space. As $T(\Omega)$ has a non-polar complement, the result follows from theorem \ref{th:pOT} above.
\end{proof}

Now we seek to extend holomorphic functions defined on hypersurfaces that are seen as complex analytic subvarieties of the pseudoconvex domains with estimates. The main ingredient in the next result is the following result of Ohsawa (see Theorem 1.1 in \cite{ohsawa2017extension}, or Theorem 7 in \cite{boudreaux2021dimension}): 
\begin{Theorem}[\textbf{Ohsawa '17}]
Let $G \subset \mathbb{C}^M$ be pseudoconvex, $\phi \in PSH(G)$, and $g \in \mathcal{O}(G)$ be such that $0 \in g(G)$, and $\mathbb{C} \setminus g(G)$ is non-polar. Moreover, suppose the restriction of $dg$  on any irreducible component of $Y=g^{-1}(\{0\})$ is not identically zero. Then for any holomorphic $(M-1)$-form $F$ on $Y_0:= Y \setminus \text{sing}(Y)$ with
$$
\left| \int_{Y_0} e^{-\phi} F \wedge \bar F \right| < \infty,
$$
there exists a holomorphic $M$-form $\tilde{F}$ on $G$ such that $\tilde{F} = dg \wedge F$ at any point of $Y_0$ and
$$
\left| \int_G e^{-\phi} \tilde{F} \wedge \bar{\tilde{F}} \right| \leq 2 \pi \left| \int_{Y_0} e^{-\phi} F \wedge \bar F \right|.
$$
\end{Theorem}

As an application of Ohsawa's theorem certain holomorphic functions defined on the non-singular subset $Y_0$ of a co-dimension one complex analytic subvariety $Y$ of $G$ can be extended to a holomorphic function on $G$. We denote the volume form on $\mathbb{C}^M$ by $dz\wedge d\bar{z}$.
\begin{Theorem}\label{th:pOhsawa1}
Suppose $1 \leq p < 2$. Let $G \subset \mathbb{C}^M$ be a pseudoconvex domain, $\phi$ be a plurisubharmonic function, and $g$ be a holomorphic function such that $\mathbb{C} \setminus g(G)$ is non-polar, $0\in g(G)$. Furthermore suppose there are no irreducible components of $Y = \{g^{-1}(0)\}$ on which $dg$ is identically zero, and $Y_0=Y \setminus sing(Y)$. Then for each $f \in \mathcal{O}(Y_0)$ with $f i\lrcorner dz$ being a holomorphic $(M-1)$-form on $Y_0$ satisfying
$$
\int_{Y_0} |f|^p e^{-p \phi} i \lrcorner dz \wedge \overline{i \lrcorner dz} < \infty 
$$
there exists a holomorphic extension $F \in \mathcal{O}(G)$ such that 
$$ 
F dz = f i \lrcorner dz \wedge dg
$$
on $Y_0$, and
$$
\int_{G} |F|^p e^{-p \phi} \,d \lambda^{2M} \leq 2 \pi \int_{Y_0} |f|^p e^{-p \phi} \, i \lrcorner dz \wedge \overline{i \lrcorner dz},
$$
where $i:Y_0 \hookrightarrow G$ is the inclusion, and $i \lrcorner dz$ is the interior product.
\end{Theorem}

\begin{proof}
Pick a point $z_0 \in Y_0$, and suppose $f(z_0)=1$. Suppose $\{G_k \subset G: k \in \mathbb{N}\}$ is a family of bounded pseudoconvex domains such that
\begin{itemize}
    \item $z_0 \in G_k$,
    \item $\overline{G_k} \subset G_{k+1}$,
    \item $\cup_k G_k = G$.
\end{itemize}

Let $\phi_k \in \mathcal{C}^{\infty}(\overline{G_k}) \cap PSH(G_k)$ be smooth plurisubharmonic functions on $G_k$ that decrease $\phi_k(z) \searrow\phi(z)$ to the plurisubharmonic function $\phi$ pointwise on $G$. Define
$$
E^p(G_k, p\phi_k):= \{F \in \mathcal{O}(G_k): \int_{G_k} |F|^p e^{-p\phi_k} \,d \lambda^{2M} < \infty, F|_{Y_0} \equiv f \}.
$$
Let us denote the weighted $L^p$-norm by $\|.\|_{p,\phi_k}$.
By the theory of Oka and Cartan (see the book \cite{grauerttheory}) $f$ admits extensions to holomorphic functions on $G$. The boundedness of $G_k$ and smoothness of $\phi_k$ on $\bar G_k$ ensures that each such extension lies in $E^p(G_k, p\phi_k)$. In particular, $E^p(G_k, \phi_k)$ is non-empty. For $p \geq 1$, the Bergman inequality for weighted $p$-Bergman spaces (see proposition $2.1$ of \cite{Chen2021OnTheory}) implies that $E^p(G_k,p\phi_k)$ is a normal family. Suppose $F_k \in E^p(G_k, \phi_k)$ is a minimal element with respect to $\|.\|_{p,\phi_k}$, i.e.
$$
\|F_k\|_{p,\phi_k} \leq \|F\|_{p, \phi_k},
$$
for all $F \in E^p(G_k, \phi_k)$. Note that by proposition $2.5$ of \cite{Chen2021OnTheory} such minimal elements of weighted $p$-Bergman spaces of bounded domains $G_k$ are unique. So, for all $h \in A^p(G_k, p\phi_k)$ with $h \equiv 0$ on $Y_0 \cap G_k$ the minimality of $F_k$ implies
$$
\|F_k \|_{p,\phi_k} \leq \|F_k + th\|_{p, \phi_k},
$$
for all $t \in \mathbb{C}$. Moreover, as the $F_k$ is minimal, we obtain
\begin{equation}\label{eq:L^pmin}
\begin{aligned}
\left. \frac{d}{dt}\right|_{t=0} \left( \int_{G_k} |F_k + th|^p e^{-p \phi_k} \,d \lambda^{2M} \right) &= {\frac{p}{2}}  \int_{G_k} |F_k|^{p-2} h \bar{F}_k e^{-p \phi_k} \,d \lambda^{2M} =0  
\end{aligned}  
\end{equation}
This motivates us to consider the space
$$
E^2(G_k, \hat{\phi}_k):= \{F\in \mathcal{O}(G_k): \int_{G_k} |F|^2 e^{-p \phi_k -(2-p) \log |F_k|} \,d \lambda^{2M} < \infty, F|_{Y_0} \equiv f\},
$$
where we denote the adjusted weight by $\hat{\phi}_k = p \phi_k + (2-p) \log |F_k|$, and the weighted $L^2$-norm by $\|.\|_{2,\hat{\phi}_k}$. Applying Ohsawa's theorem on the co-dimension one complex analytic subvariety $Y \cap G_k =\{(g|_{G_k})^{-1}(0)\}$ of the domain $G_k$, there exists $F \in \mathcal{O}(G_k)$ satisfying
$$
\int_{G_k} |F|^2 e^{-p \phi_k -(2-p) \log |F_k|} \,d \lambda^{2M} \leq 2 \pi \int_{Y_0 \cap G_k} |f|^p e^{-p \phi_k} \,i \lrcorner dz \wedge \overline{i\lrcorner dz}. 
$$
Thus $E^2(G_k, p\phi +(2-p) \log |F_k|)$ is non-empty. Since $\hat{\phi}_k$ is an admissible weight (see theorem 3.1 in \cite{pasternak1990dependence}), the Bergman inequality allows us to consider minimal elements of $E^2(G_k, \hat{\phi}_k)$ with respect to the weighted norm $\|.\|_{2, \hat{\phi}_k}$. Say, $F_0$ is the minimal element in $E^2(G_k, \hat{\phi}_k)$.

Firstly, we note that $F_0 \in E^p(G_k, \phi_k)$ since
\begin{equation}\label{eq:impH\"older}
\int_{G_k} |F_0|^p e^{-p\phi_k} \,d \lambda^{2M} \leq \left( \int_{G_k} |F_0|^2 |F_k|^{p-2} e^{-p\phi_k} \,d \lambda^{2M} \right)^{p/2} \left( \int_{G_k} |F_k|^p e^{-p \phi_k} \,d \lambda^{2M} \right)^{1-p/2}.
\end{equation}
The minimality of $F_0$ implies
\begin{equation}\label{eq:L^2min}
\left. \frac{d}{dt}\right|_{t=0} \left( \int_{G_k} |F_0 + th|^2 |F_k|^{p-2} e^{-p \phi_k} \,d \lambda^{2M} \right) = \int_{G_k} h \bar{F}_0 |F_k|^{p-2} e^{-p \phi_k} \,d \lambda^{2M} =0.  
\end{equation}
By equations (\ref{eq:L^pmin}) and (\ref{eq:L^2min}), we obtain
\begin{equation}
\begin{aligned}
\int_{G_k} |F_k|^{p-2} h(\overline{F_k - F_0}) e^{-p\phi_k} \,d \lambda^{2M} &= 0,
\end{aligned}
\end{equation}
and setting $h= F_k -F_0$ in a sufficiently small neighborhood of $z_0$ in $G_k$ this forces $F_k = F_0$ and thus on entire $G_k$. Hence, on the bounded pseudoconvex domain $G_k$ we obtain
\begin{equation}\label{eq: OTbdd}
\begin{aligned}
\int_{G_k} |F_k|^p e^{-p\phi_k} \,d \lambda^{2M} &\leq 2 \pi \int_{Y_0} |f|^p e^{-p \phi} \,i\lrcorner dz \wedge \overline{i\lrcorner dz}.
\end{aligned}
\end{equation}
Note that the right hand side of (\ref{eq: OTbdd}) is an uniform bound independent of $G_k$.

To complete the proof, we employ a standard diagonal sequence argument. By construction $F_k$ are the minimal elements of each $E^p(G_k, p\phi_k)$. Consider the sequence $(F_n)_{n \geq k} \subset E^p(G_k, \phi_k)$. Note that
\begin{align*}
\int_{G_k} |F_{k+r}|^p e^{-p\phi_k} \,d \lambda^{2M} &\leq \int_{G_k} |F_{k+r}|^p e^{-p \phi_{k+r}} \,d \lambda^{2M} \\
&\leq \int_{G_{k+r}} |F_{k+r}|^p e^{-p \phi_{k+r}} \,d \lambda^{2M} \\
&\leq 2\pi \int_{Y_0 \cap G_{k+r}} |f|^p e^{-p \phi_{k+r}} \,i\lrcorner dz \wedge \overline{i\lrcorner dz} \\
&\leq 2\pi \int_{Y_0} |f|^p e^{-p \phi} \,i\lrcorner dz \wedge \overline{i\lrcorner dz}.
\end{align*}
Using the Bergman inequality for $p$-Bergman spaces, the family $(F_n)_{n \geq k} \subset E^p(G_k, p \phi_k)$ forms a normal family. Hence, $(F_n)_{n \geq k}$ has a subsequence $(F^k_n)_{n \in \mathbb{N}}$ convergent uniformly on compact sets inside $G_k$. In the next step, inside $E^p(G_{k+1}, \phi_{k+1})$ we consider the sequence $(F^k_n)_{n \in \mathbb{N}}$, and again due to the Bergman inequality $(F^k_n)_{n \in \mathbb{N}}$ forms a normal family allowing us to extract a convergent subsequence $(F^{k+1}_n)_{n \in \mathbb{N}} \subset E^p(G_{k+1}, p \phi_{k+1})$ of $(F^k_n)_{n \in \mathbb{N}}$. In particular, for each $r \in \mathbb{N}$, we denote $(F^r_n)_{n \in \mathbb{N}} \subset E^p(G_r, p \phi_r)$ to be a normal family of minimal elements that admits a convergent subsequence converging uniformly on compact subsets inside $G_r$. Continuing in this manner we obtain the diagonal subsequence which is a subsequence of $(F^r_n)_{n \in \mathbb{N}}$ for every $r \in \mathbb{N}$. We denote the diagonal sequence by $(F^{\infty}_n)_{n \in \mathbb{N}}$. 

Define $\hat{F}_k$, complex valued on $G$ as
\begin{align*}
\hat{F}_k &\equiv F_k, \hspace{1mm} \text{on} \hspace{1mm} G_k, \\
&= 0, \hspace{1mm} \text{on} \hspace{1mm} G \setminus G_k,
\end{align*}
for all $k \in \mathbb{N}$. Similarly, by $\hat{F}^{\infty}_k$ we mean complex-valued functions that are defined as constant zero outside of the $G_{k_{\ell}}$ where it is holomorphic. Here we are assuming $F^{\infty}_k$ is the minimal element of $E^p(G_{k_{\ell}}, p \phi_{k_{\ell}})$. It again follows that
$$
\int_G |\hat{F}^{\infty}_{k}|^p e^{-p \phi_{k_{\ell}}} \,d \lambda^{2M} \leq 2\pi \int_{Y_0} |f|^p e^{-p \phi} \,i\lrcorner dz \wedge \overline{i\lrcorner dz},
$$
making the diagonal sequence a normal family. Suppose $F^{\infty}$ is a subsequential limit of the diagonal sequence $(\hat{F}^{\infty}_n)_{n \in \mathbb{N}}$. Being a uniform limit of holomorphic functions on compact subsets of $G$, the function $F^{\infty}$ must be holomorphic in $G$. Using DCT it follows
$$
\lim_{k \to \infty} \int_{G} |\hat{F}^{\infty}_k|^p e^{-p \phi_{k_{\ell}}} \,d \lambda^{2M} = \int_G |F^{\infty}|^p e^{-p\phi} \,d \lambda^{2M} \leq 2\pi \int_{Y_0} |f|^p e^{-p \phi} \,i\lrcorner dz \wedge \overline{i\lrcorner dz},
$$
proving $F^{\infty} \in A^p(G,p\phi)$ to be a holomorphic extension of $f \in A^p(Y_0, p \phi)$.

Finally, we argue that $F^{\infty}$ is indeed the minimal element of $E^p(G, p\phi)$. Suppose $F^0$ is the minimal element of $E^p(G,p\phi)$. Then
$$
\int_{G_{k_{\ell}}} |F^{\infty}_k|^p e^{-p \phi_{k_{\ell}}} \,d \lambda^{2M} \leq \int_{G_{k_{\ell}}} |F^0|^p e^{-p \phi} \,d \lambda^{2M} \leq \int_G |F^0|^p e^{-p\phi} \,d \lambda^{2M} \leq \int_G |F^{\infty}|^p e^{-p\phi} \,d \lambda^{2M},
$$
for all $k_{\ell} \in \mathbb{N}$ and by DCT we have $\|F^0\|_{p,\phi} = \|F^{\infty}\|_{p,\phi}$ on $G$.
\end{proof}

The following is an obvious consequence of theorem \ref{th:pOhsawa1} above.
\begin{Corollary}\label{cor:Ohsawa1}
Let $1\leq p \leq 2$, $G \subset \mathbb{C}^M$ be a pseudoconvex domain, $g \in \mathcal{O}(G)$  with $\mathbb{C} \setminus g(G)$ being non-polar and $dg$ is not identically zero on any irreducible component of $g^{-1}(\{0\})=Y$, and $Y_0=Y \setminus sing(Y)$. Then for any $\phi \in PSH(G)$, we have $\dim A^p(Y_0,\phi) \leq \dim A^p(G,\phi)$. 
\end{Corollary}

In the following, we prove another $L^p$-analog of an extension theorem of Ohsawa (theorem $4.1$, \cite{ohsawa2017extension}).
\begin{Theorem}\label{th:pOhsawa2}
Suppose $G \subset \mathbb{C}^M$ be pseudoconvex, and $\phi \in PSH(G)$. For any $\alpha > 0$ and for any holomorphic $f$ on $G'= G \cap \{z_M =0\}$, there exists $\tilde{f} \in \mathcal{O}(G)$, an extension of $f$ satisfying
$$
\int_G |\tilde{f}|^p e^{-\phi- \alpha |z_M|^2} \,d \lambda^{2M} \leq \frac{\pi}{\alpha} \int_{G'} |f|^p e^{-\phi} \,d \lambda^{2(M-1)}.
$$
\end{Theorem}

\begin{proof}
Fix $\alpha > 0$. Pick $z_0 \in G'$. To get started, we consider a smooth plurisubharmonic regularization of $\phi$ as in the proof of Theorem \ref{th:pOT} with exhaustion domains $\set{G_k}$ and smooth plurisubharmonic functions $\set{\phi_k}$ that decrease to $\phi$ uniformly on compact sets. Let $f \in \mathcal{O}(G')$ be such that $f(z_0)=1$, and
$$
\int_{G'} |f|^p e^{-\phi} \,d \lambda^{2(M-1)} < \infty.
$$
We denote the intersection $G_j \cap G'$ by $G'_j$, for all $j \in \mathbb{N}$. Define the spaces
$$
E^p(G_k,\phi_k):= \set{F \in \mathcal{O}(G_k): \int_{G_k} |F(z)|^p e^{-\phi_k(z)- \alpha|z_M|^2} \,d \lambda^{2M} <\infty, F|_{G'_k} \equiv f},
$$
for all $k \in \mathbb{N}$. Using Oka-Cartan (see \cite{grauerttheory}), there exists holomorphic extensions of $f$ in $G_k$, and since $\phi_k$ is smooth on $\overline{G}_k$, each such extension lie in $E^p(G_k, \phi_k)$. Thus, $E^p(G_k, \phi_k) \neq \emptyset$. By the Bergman inequality of $p$-Bergman spaces (see proposition $2.1$, \cite{Chen2021OnTheory}), $E^p(G_k,\phi_k)$ is a normal family and thus admits a minimal element with respect to the weighted norm  $\|\cdot\|_{p,\phi_k}:= \int_{G_k} |\cdot|^p e^{-\phi_k - \alpha |z_M|^2} \,d \lambda^{2M}$. We denote the minimal extension of $f$ in $E^p(G_k, \phi_k)$ by $\hat{F}_{k}$.

Define 
$$
\hat{E}^p(G_k,\phi_k):= \set{F \in E^p(G_k,\phi_k): \int_{G_k}|F|^p e^{-\phi_k - \alpha|z_M|^2} \,d \lambda^{2M} \leq \frac{\pi}{\alpha} \int_{G'_k} |f|^p e^{-\phi_k} \,d \lambda^{2(M-1)}}.
$$
Firstly, we want to show that $\hat{E}^p(G_k,\phi_k)$ is non-empty for every $k \in \mathbb{N}$ and $\alpha>0$. In fact, we shall show $\hat{F}_k$, the minimal element of $E^p(G_k, \phi_k)$ must lie in $\hat{E}^p(G_k, \phi_k)$.

To that end, let us define 
\begin{align*}
E^2(G_k, \phi_k+(2-p) \log |\hat{F}_k|):= \{F \in \mathcal{O}(G_k): F|_{G'_k} \equiv f, &\int_{G_k} |F|^2 e^{-\phi_k -(2-p) \log |\hat{F}_k| - \alpha|z_M|^2} \,d \lambda^{2M} \\
&\leq \frac{\pi}{\alpha} \int_{G'_k} |f|^2 e^{-\phi_k - (2-p) \log |\hat{F}_k|} \,d \lambda^{2(M-1)} \}.
\end{align*}
Now due to Ohsawa's theorem $4.1$ in \cite{ohsawa2017extensionVIII}, the above space $E^2(G_k, \phi_k + (2-p) \log |\hat{F}_k|)$ is non-empty. Pick, $g \in E^2(G_k, \phi_k + (2-p) \log |\hat{F}_k|)$. So, we have
$$
    \int_{G_k} |g|^2 |\hat{F}_k|^{p-2} e^{-\phi_k - \alpha|z_M|^2} \,d \lambda^{2M} \leq \frac{\pi}{\alpha} \int_{G'_k} |f|^p e^{-\phi_k} \,d \lambda^{2(M-1)}.
$$
And using reverse H\"older's inequality with $\frac{2}{p} \in (1, \infty)$, we see
\begin{align*}
    \left( \int_{G_k} (|g|^2)^{\frac{p}{2}} e^{-\phi_k - \alpha|z_M|^2} \,d \lambda^{2M} \right)^{\frac{2}{p}} &\left( \int_{G_k} (|\hat{F}_k|^{p-2})^{\frac{-p}{2-p}} e^{-\phi_k - \alpha|z_M|^2} \,d \lambda^{2M} \right)^{-(\frac{2}{p}-1)} \\
    &\leq \int_{G_k} |g|^2 |\hat{F}_k|^{p-2} e^{-\phi_k - \alpha |z_M|^2} \,d \lambda^{2M},
\end{align*}
Since both the above integrals on the left hand side of the inequality are non-zero and finite, the minimality of $\hat{F}_k$ then implies
$$
\int_{G_k} |\hat{F}_k|^p e^{-\phi_k - \alpha|z_M|^2} \,d \lambda^{2M} \leq \int_{G_k} |g|^p e^{-\phi_k -\alpha|z_M|^2} \,d \lambda^{2M},
$$
and combining with the above, we obtain
$$
\int_{G_k} |\hat{F}_k|^p e^{-\phi_k -\alpha|z_M|^2} \,d \lambda^{2M} \leq \int_{G_k} |g|^2 |\hat{F}_k|^{p-2} e^{-\phi_k-\alpha|z_M|^2} \,d \lambda^{2M} \leq \frac{\pi}{\alpha} \int_{G_k} |f|^p e^{-\phi_k} \,d \lambda^{2(M-1)}.
$$
In other words, $\hat{F}_k \in \hat{E}^p(G_k, \phi_k)$.

Now, as usual, we shall proceed with the standard diagonal sequence argument as follows. By construction each $\hat{F}_k$ are minimal elements of each $\hat{E}^p(G_k, \phi_k)$. First note that, the sequence $(\hat{F}_{\ell})_{\ell \geq k} \subset E^p(G_k,\phi_k)$. Indeed, we have
\begin{align*}
    \int_{G_k} |\hat{F}_{k+r}|^p e^{-\phi_k - \alpha |z_M|^2} \,d \lambda^{2M} &\leq \int_{G_{k+r}} |\hat{F}_{k+r}|^p e^{-\phi_{k+r}-\alpha|z_M|^2} \,d \lambda^{2M} \\
    &\leq \frac{\pi}{\alpha} \int_{G'_{k+r}} |f|^p e^{-\phi_{k+r}} \,d \lambda^{2(M-1)} <\infty.
\end{align*}
The last step is due to the fact that the integral on $G'_{k+r}$ is finite and bounded above by $\int_{G'} |f|^p e^{-\phi} \,d \lambda^{2(M-1)}$. Here $r \in \mathbb{N}$, and $C>0$ is some constant.
Also $(\hat{F}_{\ell})_{\ell \geq k} \subset E^p(G_k, \phi_k)$ constitutes a normal family since each $\hat{F}_\ell$ w.r.t. the $\|.\|_{p,\phi_k}$-norm is uniformly bounded:
$$
\int_{G_k} |\hat{F}_k|^p e^{-\phi_k- \alpha |z_M|^2} \,d \lambda^{2M} \leq \frac{\pi}{\alpha} \int_{G'} |f|^p e^{-\phi} \,d \lambda^{2(M-1)}.
$$
By Montel $(\hat{F}_{\ell})_{\ell \geq k}$ has a convergent subsequence inside $E^p(G_k, \phi_k)$, we denote it by $(\hat{F}_{k,\ell})_{\ell \in \mathbb{N}}$. So, each $\hat{F}_{k,\ell}$ is a minimal element of some $E^p(G_{k+r}, \phi_{k+r})$. Let us say, they are minimal elements in $E^p(G_{k+r_k}, \phi_{k+r_k})$. 

In the next step inside $E^p(G_{k+1}, \phi_{k+1})$, we consider the convergent subsequence $(\hat{F}_{k,\ell})_{\ell \in \mathbb{N}} \subset E^p(G_{k+1}, \phi_{k+1})$ from the previous step and apply Montel's theorem to extract a convergent subsequence $(\hat{F}_{k+1, \ell})_{\ell \in \mathbb{N}} \subset E^p(G_{k+1}, \phi_{k+1})$. Continuing in this manner, we obtain the so called diagonal subsequence $(F^{\infty}_n)_{n \in \mathbb{N}}$ that is a convergent subsequence of all $(\hat{F}_{j,\ell})_{\ell \in \mathbb{N}}$ inside any $E^p(G_j, \phi_j)$, where the convergence is considered on compact subsets on applicable $G_j$.

Assuming $F^{\infty}$ to be the limit of the diagonal subsequence $(F^{\infty}_n)_{n \in \mathbb{N}}$, applying DCT we obtain
\begin{align*}
\lim_{k \to \infty} \int_{G_{k+r_k}} |F^{\infty}_k|^p e^{-\phi_{k+r_k}- \alpha|z_M|^2} \,d \lambda^{2M} &= \int_{G_{k+r_k}} |F^{\infty}|^p e^{-\phi-\alpha|z_M|^2} \,d \lambda^{2M} \\
&\leq \frac{\pi}{\alpha} \int_{G'_{k+r_k}} |f|^p e^{-\phi} \,d \lambda^{2(M-1)} \\
&\lesssim \int_{G'_{k+r_k}} |f|^p e^{-\phi_{k+r_k}} \,d \lambda^{2(M-1)},
\end{align*}
proving $F^{\infty} \in E^p(G_{k+r_k},\phi_{k+r_k})$, for any $k$. Define, on $G$
\begin{align*}
    F_k(z) &= F^{\infty}_k(z),\hspace{1.5mm}z \in G_{k+r_k} \\
    &=0,\hspace{1.5mm}z \in G \setminus G_{k+r_k}.
\end{align*}
Clearly, $F_k$ converges to $F^{\infty}$ uniformly on compact subsets of $G$, and thus is holomorphic in $G$. Moreover, uncovering the definition of $E^p(G_{k+r_k}, \phi_{k+r_k})$ we see
\begin{align*}
    \int_G |F^{\infty}|^p e^{-\phi - \alpha|z_M|^2} \,d \lambda^{2M} &= \lim_{k \to \infty} \int_{G} |F_k|^p e^{-\phi_{k+r_k}-\alpha|z_M|^2} \,d \lambda^{2M} \\
    &= \lim_{k \to \infty} \int_{G_{k+r_k}} |F^{\infty}_k|^p e^{-\phi_{k+r_k} - \alpha|z_M|^2} \,d \lambda^{2M} \\
    &\leq \frac{\pi}{\alpha} \int_{G'_{k+r_k}} |f|^p e^{-\phi} \,d \lambda^{2(M-1)} \\
    &\leq \frac{\pi}{\alpha} \int_{G'} |f|^p e^{-\phi} \,d \lambda^{2(M-1)}.
\end{align*}
This completes the proof.
\end{proof}

Enforcing regularity conditions on the weight as in theorem \ref{th:pOhsawa2}, we can compare the dimensions of weighted $p$-Bergman spaces of $L\cap G$ and $G$.
\begin{Corollary}\label{cor:Ohsawa2}
Let $1 \leq p < 2$, $G \subset \mathbb{C}^M$ be pseudoconvex, $\phi \in PSH(G) \cap \mathcal{C}^2(G)$ be such that there exists a hyperplane $L \subset \mathbb{C}^M$ with
$$
\inf_{z \in L} H_z(\phi, N_z) > 0,
$$
where $H_z$ is the Hessian at the point $z$ and $N_z$ is the unit complex normal vector field to $L$ at $z \in L$. Then 
$$
\dim A^p(L\cap G, \phi) \leq \dim A^p(G, \phi).
$$
\end{Corollary}

\begin{proof}
Without loss of generality, let, $L=\{z_M=0\}$. By the hypothesis there exists $\alpha>0$ such that $\phi(z) - \alpha |z_M|^2 \in PSH(G)$. Thus for every non-trivial $f \in A^p(L \cap G, \phi(z) - \alpha|z_M|^2)$ there exists an extension $F \in A^p(G, \phi)$.
\end{proof}

\subsection{\texorpdfstring{$p$}{p}-Bergman Spaces of Hartogs Domains}
In this subsection we will use the tools developed in the previous one to investigate the dimension of $p$-Bergman spaces of Hartogs domains $D_{\phi}(G)$ over $G \subset \mathbb{C}^M$ of the form
$$
D_{\phi}(G) := \{ (z,w) \in G \times \mathbb{C}^N: \|w\| < e^{-\phi(z)} \},
$$
where $\|.\|$ is the \textbf{euclidean norm} on $\mathbb{C}^N$. It is well-known that the pseudoconvexity of $D_{\phi}(G)$ follows from the plurisubharmonicity of $\phi$ and the pseudoconvexity of $G$. Precisely, we are interested in investigating whether the dichotomy of dimensions in Wiegerinck's conjecture holds as we vary $1 \leq p < 2$. First, we will prove a result analogous  to \cite[Theorem 2.16]{chakrabarti2024projections} for Hartogs domains.

\bp \label{prop:monomial}
Suppose that $F(z,w)$ is a holomorphic function on $D_\phi(G)$, then $F$ can be written as a power series in the second variable $w$ i.e. 
$$
F(z,w)=\sum\limits_{\a\in \N_0^N} f_\a(z) w^\a, \text{ where } f_\a\in \mathcal{O}(G),
$$
where $\N_0$ is $\mathbb{N} \cup \{0\}$.
Further for $p\in [1,\infty)$, if $F\in A^p(D_\phi(G))$, then for each $\a$, $f_\a(z)w^\alpha \in A^p(D_\phi(G)).$
\ep
\bpf 
For a fixed $z\in G$ ,  $g_z(w):=F(z, w) $ is a holomorphic function in the ball  $B(0, e^{-\phi(z)}) \subset \C^N$. Therefore $g_{z}(w)$ has a power series expansion in $B(0, e^{-\phi(z)})  $:
\[
g_{z}(w) = \sum_{\alpha \in \mathbb{N}_0^{N}} f_{\alpha}(z) \, w^{\alpha}
\]

where
\[
f_{\alpha}(z) = \frac{1}{(2\pi i)^{N}} \int_{ b_oP_{N}(0, r)} \frac{g_{z}(\xi)}{\xi^{\alpha + 1}} d\xi_{1} \, d\xi_{2} \cdots d\xi_{N}\]
and $P_{N}(0, r)$  \text{ is the} $N$\text{-polydisk of {multi-radius} } $r \leq \left(e^{\frac{-\phi(z)}{\sqrt{N}}}, \ldots, e^{\frac{-\phi(z)}{\sqrt{N}}}\right)$.

Multiply the above equation by $ r_{1}^{\alpha_{1}} \cdots r_{N}^{\alpha_{N}} $,  integrating with respect to $dr_{1} \ldots dr_{N}$ and applying H\"older's inequality, we get:

\[
|f_{\alpha}(z)| \lesssim \frac{ \left[ \mathrm{Vol}(B(0, e^{-\phi(z)})) \right]^{1/q} }{ (e^{-\phi(z)})^{2N+1} } \left( \int_{B(0, e^{-\phi(z)})} |g_{z}(w)|^{p} d\lambda^{2N}(w) \right)^{1/p}
\]
\beq\label{eq12}
 |f_{\alpha}(z)|^{p} \left| \frac{ e^{-(2N+\abs{\a})\cdot p \cdot \phi(z)} }{ e^{-\phi(z) 2Np/q} } \right| = \abs{f_\a(z)}^pe^{-2Np\phi(z)}e^{-\abs{\a}p\phi(z)}\lesssim   \int_{B(0, e^{-\phi(z)})} |F(z, w)|^{p} d\lambda^{2N}(w) .
\eeq

For a fixed $\alpha \in \mathbb{N}_0^N$,

\[
\int_{D_\phi} \left| f_{\alpha}(z) \cdot \omega^{\alpha} \right|^p \, d\lambda^{2N}(\omega) \, d\lambda^{2M}(z)
= \int_{z \in G} \int_ {\|\omega\| < e^{-\phi(z)}}
|f_{\alpha}(z)|^p |\omega^{\alpha}|^p \, d\lambda^{2N}(\omega) \, d\lambda^{2M}(z).
\]

Since $\|\omega\| < e^{-\phi}$, we have $ |\omega^{\alpha}| \leq e^{-\phi(z)\cdot |\alpha|}$

\[\int_{D_\phi} |f_{\alpha}(z) \omega^{\alpha}|^p  \, d\lambda^{2N}(\omega) \, d\lambda^{2M}(z)
\leq\int_{G} |f_{\alpha}(z)|^p e^{-\phi(z)\cdot p\cdot |\alpha|} vol(B(0, e^{-\phi(z)})) \, d\lambda^{2M}(z).
\]
Using (\ref{eq12}), we get 
\[\int_{D_\phi} |f_{\alpha}(z) \omega^{\alpha}|^p  \, d\lambda^{2N}(\omega) \, d\lambda^{2M}(z) \lesssim \int_{G} \int_{B(0, e^{-\phi(z)})} |F(z, \omega)|^p \, d\lambda^{2N}(\omega) \, d\lambda^{2M}(z).
\]
\text{Therefore} $\sum\limits_{\alpha\in \N_0^N} f_\alpha(z)w^\a\in {A}^p(D_\phi(G)) $ implies that $ f_\alpha(z)\omega^\alpha \in A^p( D_\phi(G))\
$ for all $\a\in \N_0^N$.
\epf

An immediate consequence of the proposition (\ref{prop:pSkoda}) is the infinite dimensionality of $A^p(D_{\phi}(G))$, for all $1 \leq p < 2$ whenever the base domain $G$ is bounded and pseudoconvex.
\begin{Theorem}\label{p-Hartogs:infinte}
Let $G \subset \mathbb{C}^M$ be a bounded pseudoconvex domain. Then $A^p(D_{\phi}(G))$ is infinite dimensional, for all $1\leq p < 2$. 
\end{Theorem}

\begin{proof}
Pick $n \in \mathbb{N}^N$ and consider the plurisubharmonic function $(2N + p|n|) \phi$ on $G$. Due to boundedness of $G$, we obtain
\begin{align*}\label{Hartogs-calculation}
\int_{D_{\phi}(G)} |f_n(z)|^p |w^n|^p \,d \lambda^{2(M+N)} &\lesssim \int_G |f_n(z)|^p e^{-(2N+p|n|)\phi} \,d \lambda^{2M} \\
&\leq (2\pi)^N \sup_G (1+\norm{z}^2)^{3M} \int_G \frac{|f_n(z)|^p e^{-(2N+p|n|)\phi}}{(1+\norm{z}^2)^{3M}} \,d \lambda^{2M}.  
\end{align*}
Given $N, n$ there exists non-trivial $f_n \in \mathcal{O}(G)$ due to proposition (\ref{prop:pSkoda}) for which the preceding integral is finite. It is easy to see that $\set{f_n(z)w^n \in A^p(D_\phi(G)): n\in \N^N }$ is a linearly independent set and  thus $\dim A^p(D_{\phi}(G)) = \infty$.
\end{proof}

One can construct point separating interpolating holomorphic functions $f_n$ on $G$ by an $L^p$-analog of Hormander's theorem ensuring existence of non-trivial $f_n \in A^p(G,(2N+p|n|)\phi)$ for infinitely many different $n \in \mathbb{N}^N$. Hence, we again obtain an infinite linearly independent set $\{f_n(z) w^n: n \in \mathbb{N}^N\}$ inside $A^p(D_{\phi}(G))$. It is possible to drop the requirement of boundedness of the base pseudoconvex domain $G$ in Theorem $\ref{p-Hartogs:infinte}$ if we enforce additional conditions on the weight.

\begin{Theorem}\label{th:point-separating}
Let $1 \leq p < 2$, $G \subset \mathbb{C}^M$ be a pseudoconvex domain, and $\phi$ be a negative-valued plurisubharmonic function on $G$ such that $\phi - c\|\cdot \|^2 \in PSH(G)$ for some $c>0$. Assume $U$ is an open subset of $G$ such that the Lelong number $\nu(\phi,.)=0$ on $U$. Then $A^p(D_{\phi}(G))$ is infinite dimensional and separates points in $D_{\phi}(U)$.
\end{Theorem}

\begin{proof}
This result is a corollary of theorem \ref{th:point-sep}.

Define 
$$
\hat{\phi} = \phi - \frac{M+\epsilon}{(2N+p|n|)} \log (1+\norm{\cdot}^2).
$$ 
Since $\phi(z) - c\|z\|^2 \in PSH(G)$ by assumption, it follows that $\displaystyle{c\norm{z}^2 - \frac{M+\epsilon}{(2N+p|n|)} \log (1+\|\cdot\|^2)}$ is plurisubharmonic for sufficiently large $N,|n|$. Further as the Lelong number $\nu(\phi,z)$ of $\phi$ is zero at all points $z \in U$, it follows that $\nu(\hat{\phi},z)=0$ for all $z \in U$.

Using theorem \ref{th:point-sep} from last section for the weight $(\frac{2N}{p}+|n|)\phi$, we obtain elements in $A^p(G,(\frac{2}{p}N+|n|)\phi)$ that separate points in $U$. Note, for any set of points $(p_1,w_1), \dots, (p_k,w_k) \in D_{\phi}(U)$ with $w_j \neq 0$, there exists $f_j \in A^p(G,(\frac{2N}{p}+|n|)\phi)$ such that $f_j(p_i)= \delta_{ij}$, and thus $f_jw^n(p_i,w_i)= \delta_{ij}$. Finally, to separate points of the form $(p_1, 0), \dots, (p_k, 0) \in D_{\phi}(U)$, we note that the integrals $\displaystyle{\int_{D_{\phi}(G)} |f(z)|^p \,d \lambda^{2(M+N)}}$ and $\displaystyle{\int_{G} |f(z)|^p e^{-2N\phi} \,d \lambda^{2M}}$ are comparable. Applying theorem \ref{th:point-sep}, $A^p(G,\frac{2N}{p}\phi)$ also separates points in $U$, enabling $A^p(D_{\phi}(G))$ to separate points in $D_{\phi}(U)$. 
\end{proof}

Now we start listing consequences of the $L^p$-analog of the Ohsawa-Takegoshi extension theorem established earlier and characterize $p$-Bergman spaces of $D_{\phi}(G)$ in terms of hyperplanes and then hypervarieties inside $G$.

\begin{Corollary}\label{cor:pOTH}
Let $1 \leq p < 2$. Let $G \subset \mathbb{C}^M$ be a pseudoconvex domain containing the origin, and let $T: \mathbb{C}^M \to \mathbb{C}$ be a non-zero linear mapping with $T(G)$ having a nonpolar complement in $\mathbb{C}$. Then, for any $\phi \in PSH(G)$, $A^p(D_{\phi}(G))$ is infinite dimensional whenever $A^p(D_{\phi}(G \cap ker T))$ is.
\end{Corollary}

\begin{proof}
The proof is a consequence of theorem \ref{th:pOT}. Making an affine change of coordinates, the linear map $T$ can be converted to the projection map $\pi: z \mapsto z_1$, and say $0 \in G$. This enables us to show $G \subset \pi_1(G) \times \mathbb{C}^{M-1}$, where by $\mathbb{C}^{M-1}$ we denote the last $M-1$-coordinates of $\mathbb{C}^M$. The non-polarity of the complement of $\pi_1(G)$ assures that $c(\pi_1(G),0)$ is non-zero and finite. 

To establish infinite dimensionality of the $p$-Bergman space of any $N$-circled fibered Hartogs domain $D_{\phi}(\Omega) \subset \mathbb{C}^M \times \mathbb{C}^N$ the general strategy is to show existence of infinitely many distinct holomorphic monomials $f_n(z) w^n$ for multi-indices $n \in \mathbb{N}^N$. As the integrals $\displaystyle{\int_{D_{\phi}(G)} |f_n(z)w^n|^p \,d \lambda^{2(M+N)}}$ and $\displaystyle{\int_G |f_n|^p e^{-(2N+p|n|)\phi} \,d \lambda^{2M}}$ are comparable, the problem then reduces to obtaining non-trivial holomorphic functions inside $A^p(G, (\frac{2N}{p}+|n|)\phi)$, for infinitely many different multi-indices $n$. Applying theorem \ref{th:pOT}, each non-trivial $f \in A^p(G \cap ker(\pi_1), (\frac{2N}{p}+|n|)\phi)$ has a holomorphic extension $F$ satisfying
$$
\int_{G} |F|^p e^{-(2N+p|n|)\phi} \,d \lambda^{2M} \leq \frac{4\pi}{c(\pi_1(G),0)^2} \int_{G \cap ker(\pi_1)} |f|^p e^{-(2N+p|n|)\phi} \,d \lambda^{2(M-1)}.
$$
Thus, $\dim A^p(D_{\phi}(G \cap ker (T)) \leq \dim A^p(D_{\phi}(G))$.
\end{proof}

\begin{Corollary}
Suppose $1 \leq p < 2$. Let $G \subset \mathbb{C}^M$ be a pseudoconvex domain with $g \in \mathcal{O}(G)$ with $\mathbb{C} \setminus g(G)$ being nonpolar, and $0 \in g(G)$. Moreover, suppose there are no irreducible components of $Y = \{g^{-1}(0)\}$ on which $dg$ is identically zero, and $Y_0= Y \setminus sing(Y)$. Then $\dim A^p(D_{\phi}(G))$ is infinite if there are infinitely many $n \in \mathbb{N}^N$ such that the weighted space $A^p(Y_0, (\frac{2N}{p}+|n|)\phi)$ is non-trivial.
\end{Corollary}

\begin{proof}
For each non-trivial $f \in A^p(Y_0,(\frac{2N}{p}+|n|)\phi)$, using theorem \ref{th:pOhsawa1}, there exists a non-trivial holomorphic extension $F \in A^p(G,(\frac{2N}{p}+|n|)\phi)$.
\end{proof}

Modifying the plurisubharmonic weight function $\phi$, we obtain the following condition sufficient for infinite-dimensionality of the Hartogs domain $D_{\phi}(G)$:
\begin{Corollary}\label{cor:Hess}
Let $1 \leq p < 2$, $G \subset \mathbb{C}^M$ be a pseudoconvex domain with $\phi \in PSH(G) \cap \mathcal{C}^2(G)$. Moreover, say, $L \subset \mathbb{C}^M$ is a complex hyperplane with the property that
$$
\inf_{z \in {L}} H_z(\phi, N_z) >0,
$$
where $N_z$ is the unit complex normal vector to $L$ at $z \in L$ and {$H_z(\phi, N_z)$ is the Hessian of $\phi$ evaluated along the vector $N_z$}. Then $A^p(D_{\phi}(G))$ is infinite-dimensional whenever $A^p( D_{\phi}(A \cap G))$ is.
\end{Corollary}

\begin{proof}
For each non-trivial $A^p(A\cap G, (\frac{2N}{p}+|n|)\phi)$, using corollary \ref{cor:Ohsawa2}, there exists an extension $F \in A^p(G,(\frac{2N}{p}+|n|)\phi)$.
\end{proof}

Generalizing a condition of Jucha (proposition $3.4$ in \cite{jucha2012remark}), we obtain another sufficient condition for infinite dimensionality of $A^p(D_{\phi}(G))$. Suppose $G \subset \mathbb{C}^M$ is pseudoconvex, and $\phi \in PSH(G)$. Suppose, for each $J \in \mathbb{N}$, there exists a compact set $K_J \subset G$, $u_J \in PSH(G)$ and a function $v_J$ that is upper semi-continuous in a neighborhood $G_J$ of $K_J$ satisfying the following
\begin{equation}\label{eq:generalJucha}
\begin{aligned}
u_J(z) + \left( \frac{2(M+\epsilon_J)}{p} + M \right) \log^{+}  \|z\| &\leq J \phi(z) ,\hspace{1.5mm}z \in G \setminus K_J,\\
u_J(z) &\leq J\phi(z) + v_J(z),\hspace{1.5mm}z \in K_J, 
\end{aligned}
\end{equation}
where $\epsilon_J >0$.
We shall refer to this phenomenon as \textbf{condition} (\ref{eq:generalJucha}). Note that it is not essential for $J$ to be an integer in the above condition; if the estimates hold for countably many real $J$ then that is equivalent to condition (\ref{eq:generalJucha}). 

\begin{Theorem}\label{th:generalJucha}
Let $1 \leq p <2$, $G \subset \mathbb{C}^M$ be a pseudoconvex domain, and $\phi$ be plurisubharmonic on $G$ satisfying condition (\ref{eq:generalJucha}). Then $A^p(D_{\phi}(G))$ is infinite dimensional. 
\end{Theorem}

\begin{proof}
As was previously seen, to establish infinite dimensionality of $A^p(D_{\phi}(G))$ it is enough to show $f_n(z)w^n \in A^p(D_{\phi}(G))$, for infinitely many $n$ with $f_n \in \mathcal{O}(G) \setminus \{0\}$. Our goal will be to find non-trivial holomorphic functions $f_n$ such that $\displaystyle{\int_G |f_n|^p e^{-(2N+p|n|)\phi} \,d \lambda^{2M}}$ is finite for infinitely many multi-indices $n \in \mathbb{N}^N$. Set, $J=(2N+p|n|)$ for any multi-index $n \in \mathbb{N}^N$.

Firstly, by proposition \ref{prop:pSkoda} there exists a non-trivial $f \in \mathcal{O}(G)$ such that
$$
\int_{G} \frac{|f(z)|^p e^{-pu_J(z)}}{(1+\norm{z}^2)^{M+\epsilon_J}} \,d \lambda^{2M} < \infty.
$$
Using the first hypothesis in condition (\ref{eq:generalJucha}), it is easily seen that
$$
\int_{G \setminus K_J} |f|^p e^{-(2N+p|n|)\phi} \,d \lambda^{2M} \lesssim \int_{G \setminus K_J} \frac{|f|^p e^{-pu_J}}{(1+\norm{z}^2)^{M+\epsilon_J}} \,d \lambda^{2M}.
$$

Note that $\displaystyle{\int_{B(z_0, \delta)} |f|^p e^{-pu_J} \,d \lambda^{2M}}$ is bounded near each $z_0 \in K_J$. The second hypothesis of condition (\ref{eq:generalJucha}) ${\int_{B(z_0, \delta)} |f|^p e^{-(2N+p|n|)\phi} \,d \lambda^{2M} \lesssim \int_{B(z_0, \delta)} |f|^p e^{-pu_J} \,d \lambda^{2M}}$ near each $z_0 \in K_J$. Combining, we obtain
$$
\int_G |f|^p e^{-(2N+p|n|)\phi} \,d \lambda^{2M} \lesssim \int_G \frac{|f|^p e^{-pu_J}}{(1+\norm{z}^2)^{M+\epsilon_J}} \,d \lambda^{2M} < \infty.
$$
\end{proof}
One consequence of condition (\ref{eq:generalJucha}) is the following simple observation.
\begin{Corollary}\label{cor:generalJucha}
Let $G \subset \mathbb{C}^M$ be an unbounded pseudoconvex domain, and $\phi \in PSH(G)$ be such that $\phi - c \|.\|^2 \in PSH(G)$. Then $A^p(D_{\phi}(G))$ is infinite dimensional. 
\end{Corollary}

\begin{proof}
Choose $N \in \mathbb{N}$ such that $cN > \frac{M+\epsilon}{p} + \frac{M}{2}$ and define
$$
u_J := J\phi - J c \log (1+ \|\cdot\|^2),
$$
for all $J > N$ and some $\epsilon > 0$. Note that $\|.\|^2 - \log (1+\|.\|^2)$ is plurisubharmonic and therefore so is $u_J$. Fixing $K_J = \overline{B(0,R)}$ with $R>0$ large enough for all $J$, the first hypothesis in condition (\ref{eq:generalJucha}) is satisfied. Indeed, for all $J > N$, on $G \setminus K_J$ we have
\begin{align*}
u_J &< J\phi - \left( \frac{M+\epsilon}{p} + \frac{M}{2} \right) \log (1+ \|.\|^2) \\
u_J + \left( 2\frac{M+\epsilon}{p} + M \right) \log^+ \|.\| &< J\phi - \left( \frac{M+\epsilon}{p} + \frac{M}{2} \right) \log (1+ \|.\|^2) + \left( 2\frac{M+\epsilon}{p} + M\right) \log^+ \|.\| \\
&= J \phi + \left( \frac{M+\epsilon}{p} + \frac{M}{2} \right) \log^+ \left( \frac{\|.\|^2}{1+\|.\|^2} \right) \\
&< J\phi + O(1).
\end{align*}

And as upper semicontinuous functions are bounded above on compact sets, $v_J$ can be some large constant to satisfy the second hypothesis. 
\end{proof}
Note that, earlier we have shown the infinite dimensionality of $D_{\phi}(G)$ with strictly plurisubharmonic $\phi$ in theroem \ref{th:point-sep}. This can be alternatively listed as a consequence of condition (\ref{eq:generalJucha}).

\subsection{Hartogs Domains with One-Dimensional Base}
In this subsection, we restrict our focus to Hartogs domains $D_{\phi}(G)$ when the base domain $G$ is planar. Therefore, the weight $\phi$ needs to be a subharmonic function on a domain $G \subset \mathbb{C}$. We study $p$-Bergman spaces of $D_{\phi}(G)$, for $1 \leq p \leq 2$. In this setting, Jucha \cite{jucha2012remark} obtained sufficient conditions on subharmonic weights that force the Bergman spaces to be either zero or infinite dimensional. We show that analogous conditions on subharmonic weights establish infinite dimensionality of $p$-Bergman spaces of $D_{\phi}(G)$. These analogs can be shown by suitably adapting Jucha's arguments. We still briefly sketch the arguments for the reader's convenience.

A simple application of proposition \ref{prop:pSkoda} leads us to the following observation. 
\begin{Proposition}\label{prop3.4}
Let $1 \leq p < 2$ and $\phi$ be a subharmonic function on a domain $G \subset \mathbb{C}$. Moreover, let a compact subset $K$ of $G$ and $u \in SH(G)$ be such that
\begin{equation}\label{eq:lelong}
\begin{aligned}
u(z) + \frac{2(1+\epsilon)}{p} \log^{+} |z| &\leq C + \phi (z), \text{for all} \hspace{2mm} z \in G \setminus K \\
\lfloor \nu(p\phi, z) \rfloor &\leq \nu(pu,z), \text{for all} \hspace{2mm} z \in K
\end{aligned}
\end{equation}
for some $C \in \mathbb{R}$ and some $\epsilon > 0$. Then there exists a non-trivial $f \in \mathcal{O}(G)$ satisfying
$$
\int_G |f|^p e^{-p \phi} \,d \lambda^2 < \infty.
$$
\end{Proposition}

\begin{proof}
By proposition \ref{prop:pSkoda}, there exists a non-trivial holomorphic function $f \in \mathcal{O}(G)$ such that
$$
\int_{G} \frac{|f(z)|^p e^{-p u}}{(1+\abs{z}^2)^{1+ \epsilon}} \,d \lambda^2 < \infty.
$$
For some constant $M >0$, $f$ satisfies
$$
\int_{G \setminus K} |f|^p e^{-p \phi} \,d \lambda^2 \leq M \int_{G \setminus K} \frac{|f(z)|^p e^{-p u}}{(1+\abs{z}^2)^{1+ \epsilon}} \,d \lambda^2. 
$$
Using proposition 2.2 of \cite{jucha2012remark}, the function $e^{-pu}$ is integrable near $z_0 \in G$ if and only if the Lelong number $\nu(u,z_0) < \frac{2}{p}$. Note that, $|f|^p e^{-pu}$ is integrable near each $z_0 \in K$ since $\displaystyle{\int_{G} \frac{|f(z)|^p e^{-p u}}{(1+\abs{z}^2)^{1+ \epsilon}} \,d \lambda^2}$ is finite. For $f(z_0) \neq 0$, $e^{-pu}$ is integrable near $z_0$, and $e^{-p\phi}, |f|^p e^{-p \phi}$ are integrable near $z_0$ due to the second condition on (\ref{eq:lelong}). Near points $z_0$ with $f(z_0)=0$, the integrability of $|f|^pe^{-pu}$ is equivalent to the integrability of $|z-z_0|^{pk}e^{-pu}$, where $k$ is the order of vanishing of $f$ at $z_0$. Using Skoda's theorem (or corollary $2.4$ of \cite{jucha2012remark}), this is equivalent to $\nu(u,z_0)-k < \frac{2}{p}$. The second hypothesis in condition (\ref{eq:lelong}) implies the integrability of $|f|^pe^{-p\phi}$ on $B(z_0, \delta)$. Due to compactness of $K$, 
$$
\int_G |f|^p e^{-p \phi} \,d\lambda^2 = \int_{G \setminus K} |f|^p e^{-p \phi} \,d \lambda^{2} + \int_{K} |f|^p e^{-p \phi} \,d \lambda^{2} < \infty.
$$
\end{proof}

Next, we record some observations analogous to questions considered in \cite{jucha2012remark}. Consider, $\phi_1 = \frac{1}{2}\phi$, for some $1 \leq p <2$. The Riesz measure $\frac{1}{2\pi} \Delta \phi_1$ can be decomposed as follows
\begin{equation}\label{eq:Rdecomp}
\begin{aligned}
    \frac{1}{2\pi} \Delta \phi_1 &= \sum_{j \geq 1} \alpha_j \delta_{a_j} + \mu, 
\end{aligned}
\end{equation}
where $\mu$ is a non-negative measure that is zero in countable sets, $\alpha_j = \nu(\phi_1, a_j)$, and $\delta_{a_j}$ is the Dirac delta measure. We restate the condition on the weights $\alpha_j$ of the decomposition (\ref{eq:Rdecomp}) of the Riesz measure, sufficient for infinite dimensionality of $A^p(D_{\phi}(\mathbb{C}))$, obtained by \cite{jucha2012remark}:
\begin{equation}\label{eq:condn4.2}
\begin{aligned}
    \text{there} \hspace{1.5mm} \text{exists} \hspace{1.5mm} j_1 \neq j_2 \hspace{1.5mm} \text{with} \hspace{1.5mm} \text{frac}(\alpha_{j_1}),\, & \text{frac}(\alpha_{j_2}),\, \text{frac}(\alpha_{j_1} + \alpha_{j_2}) > 0, \hspace{1.5mm} \text{or}, \\
    \text{there} \hspace{1.5mm} \text{exists} \hspace{1.5mm} j_1 < j_2 < j_3 \hspace{1.5mm} \text{with} \hspace{1.5mm}& \alpha_{j_1}=\alpha_{j_2}=\alpha_{j_3} = \frac{1}{2}. 
\end{aligned}
\end{equation}
The above will be referred as condition (\ref{eq:condn4.2}). The main observation of this subsection is the following partial $L^p$ analog of theorem $4.1$ in \cite{jucha2012remark}. 
\begin{Theorem}\label{th:suffpJucha}
Let $\phi \in SH(\mathbb{C})$ and $1 \leq p <2$. Then $A^p(D_{\phi}(G))$ is infinite dimensional whenever
\begin{itemize}
\item $\mu \not \equiv 0$, or,
\item condition (\ref{eq:condn4.2}) is satisfied and $p\alpha_{j_1}, p\alpha_{j_2}, 1$ in the first case are $\mathbb{Q}$-linearly independent.
\end{itemize}
\end{Theorem}

As we shall see, this can be proved by slightly adjusting Jucha's ideas in \cite{jucha2012remark} to apply proposition \ref{prop3.4}. So, at times we shall only explain the necessary modifications to Jucha's proofs as required. The proof consists of the next two propositions.

\begin{Proposition}\label{prop:suffpJucha1}
Let $G \subset \mathbb{C}$ be a domain, $\phi \in SH(G)$, and $1 \leq p <2$. Suppose the weights in the decomposition satisfy condition (\ref{eq:condn4.2}) and the $p \alpha_{j_1}, p\alpha_{j_2}, 1$ in the first case are $\mathbb{Q}$-linearly independent. Then $A^p (D_{\phi}(G))$ is infinite dimensional.
\end{Proposition}

\begin{proof}
Our strategy is to show existence of non-trivial holomorphic functions in the weighted $p$-Bergman spaces $A^p(G, (\frac{2N}{p} +|n|) \phi)$ for infinitely many multi-indices $n \in \mathbb{N}^N$.

Pick $z_1, z_2, z_3 \in G$ such that $\nu(\phi_1,z_i) = \alpha_i$ satisfies the condition (\ref{eq:condn4.2}), where $\phi_1 = \frac{1}{2}\phi$. By proposition \ref{prop:key1dimJucha} in the appendix, there are infinitely many $k$ such that $\sum\limits_{j=1,2,3} \text{frac}((2N + pk) \alpha_j) >1$. Picking multi-indices $n \in \mathbb{N}^N$ such that $|n|=k$ for such $k$'s, for some $\epsilon>0$ we obtain
$$
\sum_{j=1,2,3} \text{frac}((2N+p|n|)\alpha_j) > 1+ \epsilon.
$$
Consider the subharmonic function $u(z)= (2N+p|n|)\phi_1(z) - \sum_1^3 \text{frac}((2N+p|n|) \alpha_j) \log |z-z_j|$, and note that there exists $C \in \mathbb{R}$ and $\delta > 0$ such that
\begin{equation}
\begin{aligned}
    u(z) + (1+\epsilon) \log^{+} |z| \leq (2N+p|n|)\phi_1(z) + C, 
\end{aligned}
\end{equation}
for all $z \in G \setminus \cup_{j=1}^3 B(z_j, \delta)$. Using proposition \ref{prop:pSkoda} on the subharmonic function $u$, there exists a non-trivial $f \in \mathcal{O}(G)$ such that
$$
\int_{G \setminus \cup_{j=1}^3 B(z_j, \delta)} |f|^p e^{-(2N+p|n|) \phi} \,d \lambda^2 \lesssim \int_{G} \frac{|f|^p e^{-2u}}{(1+\abs{z}^2)^{1+\epsilon}} \,d \lambda^2 < \infty.
$$

Note that on $\displaystyle{\cup_{j=1}^3 B(z_j, \delta)}$ the integral $\displaystyle{\int_{\cup_{j=1}^3 B(z_j, \delta)} |f|^p e^{-2u} \,d \lambda^2}$ is bounded, and moreover
$$
\nu(u,z_j) = \nu((2N + p|n|) \phi_1,z_j) - \text{frac}((2N+p|n|)\alpha_j), \text{and } \nu(u,z) \geq \lfloor \nu((2N + p|n|) \phi_1,z) \rfloor
$$
inside $\cup_1^3 B(z_j, \delta)$. Therefore, $f(z_0) \neq 0$ implies local integrability of $e^{-2u}$ near $z_0$ and thus both $\nu(u,z_0)$ and $\nu((2N+p|n|)\phi_1,z_0)$ are smaller than $1$ due to Skoda's theorem. Otherwise if $f$ has a zero of order $\ell$ at $z_0$ then both $\nu(u,z_0) - \frac{p}{2} \ell$ and $\nu((2N+p|n|)\phi_1,z_0)- \frac{p}{2} \ell$ are less than $1$. This ensures that the integral $\int_{\cup_1^3 B(z_j,\delta)} |f|^p e^{-(2N+p|n|)\phi} \,d \lambda^2$ is finite.

Thus for infinitely many $n \in \mathbb{N}^N$, there exists non-trivial $f_n$ holomorphic in $G$ such that
$$
\int_{D_{\phi}(G)} |f_n(z) w^n|^p \,d \lambda^{2(1+N)}(z,w) \lesssim \int_G |f_n|^p e^{-(2N+p|n|)\phi} \,d \lambda^2 < \infty.
$$
\end{proof}

Next, we remark that the argument in proposition 4.3 in \cite{jucha2012remark} can be adapted and applying proposition \ref{prop3.4}, the following $L^p$-analog of proposition 4.3 can be proved. 
\begin{Proposition}
    Let $G \subset \mathbb{C}$ be a domain, and $\phi \in SH(G)$ be such that there exists a disc $B(z_0, \delta)$ satisfying
\begin{align*}
    \Delta \phi &\not\equiv 0, \hspace{1mm} \text{on} \hspace{1mm} B(z_0, \delta) \\
    \nu(\phi,z) &= 0, \hspace{1mm} z \in B(z_0, \delta).
\end{align*}
Then $\dim A^p(D_{\phi}(G)) = \infty$, for all $1 \leq p < \infty$.
\end{Proposition}

Besides the two sufficient conditions for infinite dimensionality of $A^p(D_{\phi}(\mathbb{C}))$, we note that absence of condition (\ref{eq:condn4.2}) and $\mu$ vanishing identically may not force triviality of $A^p(D_{\phi}(G))$ as illustrated in the example below.

\begin{example} \label{ex:onedimHartogs}
    Let $\phi=\dfrac{1}{2}\log \abs{z}+\dfrac{1}{2}\log \abs{z-1}$. Clearly $\phi$ does not satisfy condition 
    (\ref{eq:condn4.2}) and also the associated harmonic part $\mu$ in the Riesz decomposition vanishes identically. 
    Let $$
D_{\phi}(\C) := \{ (z,w) \in G \times \mathbb{C}: |w| < e^{-\phi(z)} \}.
$$
We show that $A^p(D_\phi(\C))$ is non-trivial by proving that  $f_k(z,w):=z^k(z-1)^kw^{2k+1}\in A^p(D_\phi(\C))$ for all $k\in \N$ . 
    Note that \begin{align*}
    \int_{D_\phi(\C)}\abs{z}^{pk}\abs{z-1}^{pk}\abs{w}^{p(2k+1)} d\lambda^4 &\lesssim \int_\C \abs{z}^{pk}\abs{z-1}^{pk} \abs{z}^{-(p(2k+1)+2)/2}\abs{z-1}^{-(p(2k+1)+2)/2} d\lambda^2\\
    &=\int_\C \abs{z}^{-p/2-1} \abs{z-1}^{-p/2-1} d\lambda^2 
    \end{align*}
    
   Therefore it suffices to prove the local integrability of $ \abs{z \cdot(z-1)}^{-p/2}$ near $0,1,\infty$.
   \begin{itemize}
       \item[-] \normalfont\text{Local integrability near $\infty$: }
       \begin{align*} \int_{\abs{z}>R}  \abs{z \cdot(z-1)}^{-p/2-1} d\lambda^2 \simeq \int_{r=R}^\infty r^{-1-p}dr <\infty \text{ as } 1\leq p<2.
    \end{align*} 

    \item[-] \normalfont\text{Local integrability near $0$: } \begin{align*} \int_{\abs{z}<\e}  \abs{z}^{-p/2-1} d\lambda^2 \simeq \int_{r=0}^\e r^{-p/2}dr <\infty  \text{ as } 1\leq p<2.
    \end{align*} 
    \item[-] Similarly the function is {locally integrable near $1$}.
   \end{itemize} 
   Thus $f_k\in A^p(D_\phi(\C))$ for all $k\in \N$ and hence it is infinite dimensional.
\end{example}

Finally, we present a sufficient condition for triviality of $A^p(D_{\phi}(\mathbb{C}))$, where $1 \leq p < 2$. The underlying idea behind the result is that the necessary condition for local integrability and integrability near infinity contradict each other. 
\begin{Proposition}
Suppose $1 \leq p < 2$ and $\phi \in SH(\mathbb{C})$ such that $\nu(\phi, z_0) \geq 1$ for some $z_0 \in \mathbb{C}$. Further, suppose, outside some compact set
$$
\phi(z) \leq \log |z| + C, 
$$
for some $C > 0$. Then $A^p(D_{\phi}(\mathbb{C}))$ is trivial.
\end{Proposition}

\begin{proof}
Without loss of generality assume $\nu(\phi,0)\geq 1$. Suppose $F(z,w)=\sum_{n\in \N_0^n} f_n(z) w^n  \in A^p(D_{\phi}(\mathbb{C}))$, where $f_n$ is an entire function. Using proposition \ref{prop:monomial},  $f_n(z)w^n \in A^p(D_{\phi}(\mathbb{C}))$ for all  multi-index $n \in \mathbb{N}_0^N$. Observe that 
\begin{align*}
    \norm{f_n(z)w^n}_{A^p(D_{\phi}(\C))}&=\int_{\C}\abs{f_n(z)}^p\int_{\abs{w}<e^{-\phi(z)}} \abs{w^n}^pd\lambda^{2N} d\lambda^2\\
    & \simeq \int_\C \abs{f_n(z)}^p e^{-(p\,\abs{n}+2N)\phi(z)} d\lambda^2
\end{align*}

Suppose $\nu(\phi, 0) = s$. Then $\phi(z) - s \log |z|$ is bounded near $0 \in \mathbb{C}$. Assuming $f_n$ has a zero of degree $m$ at zero and $f_n (z) = z^m g(z)$, the local integrability 
$$
\int_{|z|< \epsilon} |f_n|^p e^{-(2N+p|n|)\phi} \,d \lambda^2 \lesssim \int_{|z|< \epsilon} |f_n|^p |z|^{-s(2N+p|n|)} \,d \lambda^2 \simeq \int_{r=0}^{\epsilon} r^{pm- s(2N+p|n|)+1} \,dr < \infty
$$
near $0 \in \mathbb{C}$ forces $m  >s(\frac{2N}{p}+|n|) -\frac{2}{p}$.

Similarly the integrability near infinity
\begin{align*}
\int_R^{\infty} |g(0)|^p r^{pm-(2N+p|n|)+1} \,dr &\lesssim
\int_{\abs{z}>R} \abs{g}^p \abs{z}^{pm} \abs{z}^{-(p\,\abs{n}+2N)} d\lambda^2 \\
&\lesssim \int_{\abs{z}>R} \abs{f_n(z)}^p e^{-(p\,\abs{n}+2N)\phi(z)}d\lambda^2 < \infty
\end{align*}
requires $m <(\frac{2N}{p} + |n|) -\frac{2}{p}$.
Thus $g$ must be identically zero and $f_n\equiv 0$ for all $n$ forcing $F\equiv 0$. 
\end{proof}

\subsection{Balanced Domains}
Let $h: \mathbb{C}^n \to [0, \infty)$ be an upper semi-continuous function. We call $h$ homogeneous if $h(\lambda z) = |\lambda| h(z)$, for all $\lambda \in \mathbb{C}$. A \emph{balanced domain }associated with upper semi-continuous homogeneous $h$ is denoted by $D_h$ and is defined as
$$
D_h := \{w \in \mathbb{C}^n: h(w) < 1\}.
$$
A balanced domain is called \emph{elementary} whenever $h: \mathbb{C}^n \to [0, \infty)$ is an upper semi-continuous homogeneous function of the form
$$
h(z) = \prod_{j=1}^k |A_j (z)|^{t_j},
$$
where $A_j: \mathbb{C}^n \to \mathbb{C}$ are non-trivial linear maps and $t_j > 0$ with $\sum t_j =1$. The dimensionality of the Bergman spaces of unbounded balanced domains was studied by Pflug and Zwonek in an insightful paper \cite{pflug2017h}. In a significant finding, they establish the dichotomy of the dimension of the Bergman spaces of balanced domains inside $\mathbb{C}^2$, and show that balanced domains have trivial Bergman spaces iff they are elementary. We establish a sufficient condition for triviality of balanced domains in $\mathbb{C}^2$.

Firstly, note that for any balanced domain $D_h \subset \mathbb{C}^n$, the set $D_h \setminus \{z_n =0\}$ is biholomorphic to 
$$
D_{\phi}(\mathbb{C}^{n-1}):= \set{ w \in \mathbb{C}^n: |w_n| < \exp(-\log h(w',1))},
$$
the map
\begin{equation} \label{eq:balancedbiholo}
\psi:(z', z_n) \mapsto \(\frac{z'}{z_n}, z_n\)
\end{equation}
being a biholomorphism, where $w'=(w_1, \dots, w_{n-1})$, $z'=(z_1, \dots, z_{n-1})$, and $\phi(w') = \log h(w',1)$. So the $p$-Bergman space $A^p(D_h \setminus \{z_n=0\})$ is isomorphic to the weighted $p$-Bergman space $A^p(D_{\phi}(\mathbb{C}^{n-1}), -2\log (|\det Jac(\psi^{-1})|))$. This enables us (and Pflug-Zwonek in theorem $7$, \cite{pflug2017h}) to use Jucha's ideas. 

\begin{Theorem}\label{th:balancedtriviality}
Suppose $1 \leq p < 2$, $D_h \subset \mathbb{C}^2$ be a balanced pseudoconvex domain, and $\phi(z) = \log h((z,1))$. If there exists $w_0 \in \mathbb{C}$ with $\nu(\phi, w_0) = 1$ then $A^p(D_h)$ is trivial.
\end{Theorem}

\begin{proof}
Using the biholomorphism (\ref{eq:balancedbiholo}) above, it can be seen that the triviality of $A^p(\mathbb{C}, (pn+4)\phi)$ for all $n \in \mathbb{N}$ implies the triviality of $A^p(D_h)$. Therefore we shall show that $A^p(\mathbb{C}, (pn+4)\phi)$ must be trivial for all $n \in \mathbb{N}$.

Suppose $f \in A^p(\mathbb{C}, (pn+4)\phi)$ for some $n \in \mathbb{N}$. Further suppose $\nu(\phi,0) =1$ without loss of generality. Assume $f$ has a zero of order $m$ at $0 \in \mathbb{C}$, i.e. $f(z)= z^mg(z)$ for some entire $g$ with $g(0)\neq 0$. Since $\phi(z) - \log |z|$ is bounded near the origin, the integrability
$$
\int_0^{\epsilon} r^{pm-(pn+4)+1} \,dr \lesssim \int_{|z|< \epsilon} |f|^p e^{-(pn+4)\phi} \,d \lambda^2 < \infty
$$
of $f$ near $0\in \mathbb{C}$ implies $m>n+\frac{2}{p}$.

Since $h$ is upper semi-continuous and the unit sphere is compact, we obtain $\phi(z) = \log \|(z,1)\| + \log h \left(\frac{(z,1)}{\|(z,1)\|}\right) \leq \log \|(z,1)\| + C$, for some $C>0$. Therefore integrability near infinity
$$
\int_R^{\infty} |g(0)|^p r^{pm-(pn+4)+1} \,dr  \lesssim \int_{|z|> R} |f|^p |z|^{-(pn+4)} \,d \lambda^2 \lesssim \int_{|z|>R} |f|^p e^{-(pn+4)\phi} \,d \lambda^2
$$
requires $m <n+\frac{2}{p}$. Therefore no non-trivial $f$ can exist in $A^p(\mathbb{C}, (pn+4)\phi)$, for any $n$. This completes the proof.
\end{proof}

In contrast to \cite[Theorem 7]{pflug2017h}, we remark that the elementary balanced domains may not have trivial $p$-Bergman spaces.
\begin{example}\label{ex:elbalanced}
Define linear maps $A_1,A_2: \mathbb{C}^2 \to \mathbb{C}$ as \begin{align*}
    A_1(z,w)&=z-2w\\
    A_2(z,w)&=z-w
\end{align*} and $h((z,w)) = |A_1(z,w)|^{\frac{1}{2}} |A_2(z,w)|^{\frac{1}{2}}$. Then $D_h:= \{(z,w): h(z,w)<1\}$ is an elementary balanced domain. Let 
$$
\phi(z)= \frac{1}{2} \log |A_1(z,1)|+\frac{1}{2} \log|A_2(z,1)|
$$
Note that $A^p(D_{\phi}(\mathbb{C}), -2 \log |z|)$ is isomorphic to $A^p(D_h)$.
    
We shall show that $(z-1)^k(z-2)^kw^{2k-1} \in A^p(D_{\phi}(\mathbb{C}),-2\log |z|)$ thereby proving the nontriviality of $A^p(D_h)$. It suffices to prove non-triviality of $A^p(\mathbb{C}, (pn+4)\phi)$ for all $n=2k-1, k \in \mathbb{N}$. Precisely, we show $f_k = (z-1)^k(z-2)^k \in A^p(\mathbb{C},(p(2k-1)+4)\phi)$. Note that $\nu(\phi,1)= \nu(\phi,2) = \frac{1}{2}$. A similar computation as example \ref{ex:onedimHartogs} shows $(z-1)^k(z-2)^k \in A^p(\mathbb{C},(p(2k-1)+4)\phi)$ for all $k \in \mathbb{N}$. 
\end{example}

\section{Weighted \texorpdfstring{$p$}{p}-Fock Spaces of Entire Functions}\label{sec:pFock}
Let $\phi$ be a plurisubharmonic function in $\mathbb{C}^n$. A weighted Fock space of entire functions in $\mathbb{C}^n$ with weight $\phi$ is denoted as $\mathcal{F}^2_{\phi}$, and defined as
$$
\mathcal{F}^2_{\phi} := \{f \in \mathcal{O}(\mathbb{C}^n): \int_{\mathbb{C}^n} |f|^2 e^{-2\phi} \,d \lambda^{2n} < \infty \}.
$$
In this section, we derive the dimensions of the $p$-analogs (calling them $p$-Fock spaces) of these spaces of entire functions, for $1 \leq p \leq 2$. A \textit{$p$-Fock space} of entire functions is thus defined as
$$
\mathcal{F}^p_{\phi} := \{f \in \mathcal{O}(\mathbb{C}^n): \int_{\mathbb{C}^n} |f|^p e^{-p\phi} \,d \lambda^{2n} < \infty\}.
$$
We obtain two sufficient conditions for the infinite dimensionality of $\mathcal{F}^p_{\phi}$, where $\phi$ is plurisubharmonic in $\mathbb{C}^n$. Suppose for each $M \in \mathbb{N}$, there exists some compact subset $K_M$ of $\mathbb{C}^n$, $u_M \in PSH(\mathbb{C}^n)$, and a function $v_M$, upper semi-continuous on a neighborhood $G_M$ of $K_M$ satisfying
\begin{equation}\label{eq:condnpFock}
\begin{aligned}
    u_M(z) + \left( \frac{2(n+\epsilon_M)}{p}  + M \right) \log^+ \|z\| &\leq \phi(z), \hspace{9.5mm} z \in \mathbb{C}^n \setminus K_M, \\
    u_M(z) &\leq \phi(z) +v_M(z) , \hspace{5.5mm} z \in K_M.
\end{aligned}
\end{equation}
The above shall be referred to as \textbf{condition} (\ref{eq:condnpFock}). 

This sufficient condition is motivated by the one obtained by Jucha in the case of Hartogs domains with one dimensional base. However, unlike the one-dimensional case local integrability of $e^{-2\phi}$ near $a$ for a plurisubharmonic $\phi$ is not equivalent to the Lelong number $\nu(\phi,a) < 1$. Only one side is true- a theorem of Skoda {\cite{skoda1972}} (see also  \cite[Lemma 5.6]{demailly_analytic_2012}) asserts Lelong number $\nu(\phi,a) < 1$ implies local integrability of $e^{-2\phi}$ near $a$. This is the primary reason we resort to pointwise estimates to handle local integrability. 

The motivation behind condition (\ref{eq:condnpFock}) is that once a non-trivial entire function $f$ (whose existence depends on the asserted $u_M \in PSH(\mathbb{C}^n)$ in condition \ref{eq:condnpFock}) is in the $p$-Fock space, then the functions of the form $z^k f$ must also lie in the $p$-Fock space, for $0 \leq k \leq M$. Here $z^k$ is the monomial $z^k_1 z^k_2 \dots z^k_n$.
\begin{Theorem}\label{th:pFock}
    Let $\phi$ be a plurisubharmonic function in $\mathbb{C}^n$ satisfying condition (\ref{eq:condnpFock}). Then $\dim \mathcal{F}^p_{\phi} = \infty$, for all $1 \leq p \leq 2$.
\end{Theorem}

\begin{proof} Pick $M \in \mathbb{N}$.
    Rewriting the condition (\ref{eq:condnpFock}), we see that
    $$
-\phi(z) + M \log^+ \|z\| \leq  -u_M(z) - \frac{2(n+\epsilon_M)}{p} \log^+ \|z\|,
    $$
    outside some compact set $K_M \subset \mathbb{C}^n$. Using the remark \ref{pSkoda:entire}, there exists a non-trivial entire function $f$ such that $\displaystyle{\int_{\mathbb{C}^n} \frac{|f(z)|^p e^{-pu_{M}(z)}}{(1+\norm{z}^2)^{n+\epsilon_M}}} d\lambda^{2n}$ is finite  and $f(z_0)=1$. Using condition (\ref{eq:condnpFock}), we see
    \begin{equation}\label{eq:outsideK_M}
\int_{\mathbb{C}^n \setminus K_M} \|z\|^{pM} |f(z)|^p e^{-p \phi} \,d \lambda^{2n} \lesssim \int_{\mathbb{C}^n \setminus K_M} \frac{|f(z)|^p e^{-p u_M(z)}}{(1+\norm{z}^2)^{n+\epsilon_M}} \,d \lambda^{2n}.
    \end{equation}
    
Since $\|z\|^k \leq \|z\|^M$ outside the unit ball and as $B(0,1) \setminus K_M$ is precompact, we obtain
$$
\int_{\mathbb{C}^n \setminus K_M} \|z\|^{pk} |f(z)|^p e^{-p\phi(z)} \,d \lambda^{2n} \lesssim \int_{\mathbb{C}^n \setminus K_M} \|z\|^{pM} |f(z)|^p e^{-p\phi(z)} \,d \lambda^{2n},
$$
for all $0 \leq k \leq M$. Inside $B(0,1)$ the integrals are comparable.   

Near each point $z_0 \in K_M$, the integral $\int_{B(z_0, \delta)} |f|^p e^{-p u_M} \,d \lambda^{2n}$ is finite, where $\delta>0$ is such that $B(z_0,\delta) \subset G_M$. Note that, the second hypothesis in condition (\ref{eq:condnpFock}) implies 
$$
\int_{B(z_0, \delta)} |f|^p e^{-p\phi} \,d \lambda^{2n} \lesssim \int_{B(z_0,\delta)} |f|^p e^{-pu_M}\,d \lambda^{2n},
$$   
and since for any $M \in \mathbb{N}$ the integrals $\displaystyle{\int_{B(z_0, \delta)} \|z\|^{pM} |f|^p e^{-p\phi} \,d \lambda^{2n}}$ and $\displaystyle{\int_{B(z_0, \delta)} |f|^p e^{-p\phi} \,d \lambda^{2n}}$ are comparable, we conclude
    $$
\int_{\mathbb{C}^n} \|z\|^{pk} |f|^p e^{-p \phi} \,d \lambda^{2n} \lesssim   
\int_{\mathbb{C}^n} \|z\|^{pM} |f|^p e^{-p \phi} \,d \lambda^{2n} \lesssim \int_{\mathbb{C}^n} \frac{|f|^p e^{-pu_M}}{(1+\norm{z}^2)^{n+\epsilon_M}} \,d \lambda^{2n} < \infty,
    $$
    for all $0 \leq k \leq M$.
Since integrability of $\|z\|^{nM}$ and $|z_1 z_2 \dots z_n|^{M}$ are equivalent using RMS-GM, $z^{M} f \in \mathcal{F}^p_{\phi}$ for each $nM$, we obtain $\dim \mathcal{F}^p_{\phi} \geq M+1$.
\end{proof}

The argument in theorem \ref{th:pFock} works seamlessly if $\phi$ is some non-plurisubharmonic function satisfying condition (\ref{eq:condnpFock}). We obtain another sufficient condition on the plurisubharmonic weight $\psi$ as an immediate consequence of condition (\ref{eq:condnpFock}) to establish infinite dimensionality of the weighted $p$-Fock space. This is analogous to the theorem $3.1$ of \cite{borichev2021dimension}.

\begin{Corollary}\label{th:pBorichev}
    Let $\phi: \mathbb{C}^n \to \mathbb{R}$ be $\mathcal{C}^2$-smooth function such that 
    \begin{itemize}
    \item for every $M>0$ the function $\phi - \phi_M$ is plurisubharmonic (or subharmonic) outside some compact subset $K_M \subset \mathbb{C}^n$,
    \item there exists $u_M \in PSH(\mathbb{C}^n)$ satisfying $u_M \leq \phi - \phi_M$ on $\mathbb{C}^n \setminus K_M$, where $\phi_M = M \log \|z\|$.
    \end{itemize}
    Then the weighted $p$-Fock space $\mathcal{F}^p_{\phi}$ is infinite dimensional, for all $1 \leq p \leq 2$.
\end{Corollary}

The following $L^p$-analog of a result of \cite{shigekawa1991spectral} (and theorem $11.20$ in \cite{haslinger2017complex}) can be shown as a consequence of theorem \ref{th:pFock}. We partly adapt Haslinger's argument to complete the proof.
\begin{Corollary}
Let $\Phi: \mathbb{C}^n \to \mathbb{R}$ be a $\mathcal{C}^2$ function, $\mu(z)$ denote the smallest eigenvalue of the Levi matrix $i \partial \bar \partial \Phi(z)$. Suppose that $\lim\limits_{|z| \to \infty} \|z\|^2 \mu(z) = \infty$. Then $\mathcal{F}^p_{\Phi}$ is infinite dimensional, for all $1\leq p \leq 2$.
\end{Corollary}

\begin{proof}
Firstly, note that since the Levi matrix is Hermitian the smallest eigenvalue $\mu(z)$ of $i \partial \bar \partial \Phi(z)$ is continuous on $z$. Since $\|z\|^2 \mu(z) \to \infty$ as $\|z\| \to \infty$, there exists $R > 0$ such that $\mu(z) >0$ whenever $\|z\|>R$. Thus there exists $K>0$ such that
$\mu(z) > -K$ on $\|z\| \leq R$. Then for any $z,w \in \mathbb{C}^n$
$$
i \partial \bar \partial \Phi (w,w)(z) \geq -K\|w\|^2.
$$

This motivates us to adjust $\Phi$ as follows to obtain a plurisubharmonic function $\hat{\Phi}$ such that the weighted $p$-Fock spaces $\mathcal{F}^p_{\Phi}$ and $\mathcal{F}^p_{\hat{\Phi}}$ are comparable. Defining $g(z) = \log (1 + 4K\|z-\zeta\|^2)$, a computation shows that the Levi matrix $i \partial \bar \partial g(w,w)(z)$ satisfies
$$
\frac{4K\|w\|^2}{(1+4K\|z-\zeta\|^2)^2} \leq i \partial \bar \partial g(w,w) (z) \leq \frac{4K \|w\|^2}{1+4K\|z-\zeta\|^2},
$$
and inside the ball $\|z-\zeta\| \leq 1/ 2\sqrt{K}$, we see $i \partial \bar \partial g(w,w)(z) \geq K\|w\|^2$ for all $z,w \in \mathbb{C}^n$. Suppose $\{B(\zeta_{\ell}, 1/2 \sqrt{K}): 1 \leq \ell \leq M\}$ is a finite cover of $\|z\| \leq R$, and define $g_{\ell}(z) = \log (1 + 4K \|z-~\zeta_{\ell}\|^2)$. Say, we work with $R>0$ large enough so that $M > \frac{2n}{p}$. Thus, defining
$$
\hat{\Phi}(z):= \Phi(z) + \sum_1^M g_{\ell}(z),
$$
we obtain our desired plurisubharmonic function to use in condition (\ref{eq:condnpFock}).

Say, the smallest eigenvalue of the Levi matrix of $\hat{\Phi}(z)$ is $\hat{\mu}(z)$. This means $\|z\|^2 \hat{\mu}(z) \to \infty$ as $\|z\| \to \infty$. For each $N \in \mathbb{N}$, there exists $R_N > 0$ such that $\hat{\mu}(z) \geq \frac{N+M+1}{\|z\|^2}$ on $\|z\| > R_N$. Also, setting $\hat{\mu}_0:= \inf \{\hat{\mu}(z): \|z\| \leq R_N\}$, we see $\hat{\mu}_0 >0$. Further, setting $k = \frac{\hat{\mu}_0}{2(N+M)}$ and
$$
u_N = \hat{\Phi}(z) -(N+M) \log (1+k\|z\|^2) + \left(N + \frac{2(n+\epsilon_N)}{p}\right) \log \sqrt{k},
$$
it can be shown that $u_N \in PSH(\mathbb{C}^n)$. Now, to compare the spaces $\mathcal{F}^p_{\Phi}$ and $\mathcal{F}^p_{\hat{\Phi}}$, note that the local integrability against the weights $\phi$ or $\hat{\phi}$ are equivalent. To handle the integrability outside a compact set, for $\ell \leq M$, we observe the following
\begin{align*}
\int_{\mathbb{C}^n} |f|^p (1+\|z\|^2)^{\ell} e^{-\Phi} \,d \lambda^{2n} &= \int_{\mathbb{C}^n} \frac{(1+\|z\|^2)^{\ell} |f|^p e^{-\hat{\Phi}}}{\prod_1^M(1+4K\|z-\zeta_j\|^2)} \,d \lambda^{2n} \\
\int_{\|z\| > R} (1+\|z\|^2)^{\ell} |f|^p e^{-\Phi} \,d \lambda^{2n} &\leq \sup_{\|z\|> R} \frac{(1+\|z\|^2)^\ell}{(1+4K\|z\|^2)^{M}} \int_{\|z\|>R} |f|^p  e^{-\hat{\Phi}} \,d \lambda^{2n} \\
&\lesssim \int_{\|z\|>R} |f|^p e^{-\hat{\Phi}} \,d \lambda^{2n}.
\end{align*}
So, for each $f \in \mathcal{F}^p_{\hat{\Phi}}$ are assigned the functions $(z_1 z_2 \dots z_n)^{\ell}f(z) \in \mathcal{F}^p_{\Phi}$, where $0 \leq p \ell \leq N$ and $z=(z_1, \dots, z_n) \in \mathbb{C}^n$. In particular, infinite dimensionality of $\mathcal{F}^p_{\hat{\Phi}}$ implies the infinite dimensionality of $\mathcal{F}^p_{\Phi}$.

Since for each $N \in \mathbb{N}$ the function $u_N \in PSH(\mathbb{C}^n)$ and $\hat{\Phi}$ satisfies condition (\ref{eq:condnpFock}) where $K_N$ is some sufficiently large ball $B(0,r_N)$, by theorem (\ref{th:pFock}) $\dim \mathcal{F}^p_{\hat{\Phi}} = \infty$. Indeed,
\begin{align*}
u_N + \left( \frac{2(n+\epsilon_N)}{p} + N \right) \log^+ \|z\| &= \hat{\Phi} + N \log^+ \left( \frac{\sqrt{k}\|z\|}{1+k\|z\|^2} \right) + \frac{2(n+\epsilon_N)}{p} \log^+ \left(\frac{\sqrt{k}\|z\|}{1+k\|z\|^2} \right) \\
&- \hat{M} \log (1+k\|z\|^2) \\
\leq C_N + \hat{\Phi},
\end{align*}
outside a ball large enough so that $\sqrt{k}\|z\| > 1$ for some constant $C_N >0$. The second hypothesis in condition (\ref{eq:condnpFock}) is easily verified.
\end{proof}

\section{Complete Reinhardt Domains with Finite Dimensional \texorpdfstring{$p$}{p}-Bergman Space}\label{sec:Reinhardt}

A domain $\O \subset \mathbb{C}^n$ is said to be Reinhardt if $$(e^{i\t_1} z_1, e^{i\t_2}z_2,\dots, e^{i\t_n}z_n ) \in \O, \text{ for all }\t_i\in \R \text{ and } (z_1,z_2,\dots, z_n)\in \O. $$ $\O$ is called complete Reinhardt if $(\lambda_1 z_1, \dots, \lambda_n z_n) \in \O$ whenever $(z_1, \dots, z_n) \in \O$, where $\lambda_j \in \mathbb{C}$ with $|\lambda_j| \leq 1$. For a Reinhardt domain $\O$, the set of  $A^p$-allowable indices is defined as, $$S_p(\O)=\set{\a\in \Z^n: e_\a:=z_1^{\a_1}\cdot z_2^{\a_2}\cdots z_n^{\a_n}\in A^p(\O)}.$$
The following result from \cite{chakrabarti2024projections} says that $\{e_\alpha: \alpha \in S_p(\O)\}$ forms a Schauder basis  for $A^p(\O)$. 
\bl\cite[Theorem 2.16]{chakrabarti2024projections} \label{Chakrabarti-schauderBasis}
Let $\O$ be a Reinhardt domain. Suppose that $e_\a\not \in S_p(\O)$, then $a_\a(f)=0$ for all $f\in A^p(\O)$. Therefore the Laurent series of $f$ is $$f(z)=\sum_{\a\in S_p(\O)} a_\a(f) z^\a$$
Therefore, if $S_p(\O)$ is a finite set, then $A^p(\O)$ is a finite dimensional space.
\el

We will use the ideas  from \cite[Lemma 1,2]{wiegerinck1984domains} and \cite[Theorem 2.16]{chakrabarti2024projections} to construct a complete Reinhardt domain $\O_k \subset \C^n$ such that dim($A^p(\O_k)$)=$k$. In particular,  we seek to construct a complete Reinhardt domain $\Omega_k$ in $\mathbb{C}^n$ whose $p$-Bergman space is $k$-dimensional and has $\set{z^{\ell}_1 z^{\ell}_2 \dots z^{\ell}_n: 1 \leq \ell \leq k}$ as its basis.

For $1 \leq j \leq n-1$, define, 
$$
X_j:=\set{z :\prod_{i \neq j}  \abs{z_i} < \frac{1}{|z_j| \log |z_j|},\, |z_j| > e, 1 < |z_i| < e, \text{for } i \neq j, j+1},
$$
and in the remaining case set, $X_n := \set{z: \prod\limits_{i \neq n} \abs{z_i} < \dfrac{1}{|z_n| \log |z_n|}, |z_n|>e, 1 < |z_i| < e, \text{for }i \neq n, 1}$.

Next, for $1 \leq j \leq n-1$, we show $z^{\ell_1}_1 z^{\ell_2}_2 \dots z^{\ell_n}_n \in A^p(X_j)$ only if $\ell_j \leq \ell_{j+1}$, and $\ell_n \leq \ell_1$ when $j=n$. This claim follows from the following straightforward computation.

\begin{align*}
    \int_{X_1} |z^{\ell_1}_1 \dots z^{\ell_n}_n|^p d \lambda^{2n} &\approx \int_{r_1=e}^{\infty} \int_{r_3=1}^e \dots \int_{r_n =1}^e \int_{r_2=0}^{\frac{1}{r_1\cdot r_3 \dots r_n \log r_1 }} \prod r^{p \ell_i +1}_i \,d r_2 \,dr_n \dots \,dr_1 \\
    &= \prod_{i \neq 1,2}^n \left( \int_1^e r^{p(\ell_i - \ell_2) -1}_i \,d r_i \right) \left( \int_{r_1=e}^{\infty}\frac{r^{p(\ell_1 - \ell_2)-1}_1}{(\log r_1)^{p \ell_2 +1}} \,d r_1 \right) \\
    &\approx \int_{r_1=e}^{\infty}\frac{r^{p(\ell_1 - \ell_2)-1}}{(\log r_1)^{p \ell_2 +1}} \,d r_1 < \infty
\end{align*}
forces $p(\ell_1 - \ell_2) < 1$. Thus $z^{\ell_1}_1 z^{\ell_2}_2 \dots z^{\ell_n}_n \in A^p(X_1)$ only if $\ell_1 \leq \ell_2$. This means $A^p(\cup_1^n X_i)$ contains all monomials $z^{\ell_1}_1 z^{\ell_2}_2 \dots z^{\ell_n}_n$ only if $\ell_1 = \dots = \ell_n = \ell$, where $\ell \in \mathbb{N}$.

In order to make the $p$-Bergman space finite dimensional, we add the following open subset
$$
B_m:= \{(z_1, \dots, z_n): |z_1|, |z_2|>1, \abs{|z_1|-|z_2|}< \frac{1}{\abs{|z_1|+|z_2|}^m}, 1< |z_j|<e, j \neq 1,2\}.
$$
Now $z^{\ell}_1 \dots z^{\ell}_n \in A^p(B_m)$ only if
\begin{align*}
    \int_{B_m} |z_1|^{p \ell} \dots |z_n|^{p \ell} \,dV &= \int_{r_j = 1}^{e} \int_{r_1,r_2 > 1} r^{p \ell +1}_1 r^{p \ell +1}_2 \dots r^{p \ell +1}_n \,dr_1 \,dr_2 \dots \,dr_n \\
    &= M \int_{r_1, r_2 >1} r^{p \ell +1}_1 r^{p \ell +1}_2 \,dr_1 \,dr_2 \\
    \text{Changing variables with } r_1=t+s, \,&r_2=t-s, \\ 
    &\approx M \int_{t=2}^{\infty} \int_{s = -\frac{1}{t^m}}^{\frac{1}{t^m}} (t^2-s^2)^{p \ell +1} \,ds \,dt \\
    &= M \int_{t=2}^{\infty} \int_{s= -\frac{1}{t^m}}^{\frac{1}{t^m}} t^{2p \ell + 2} (1- \frac{s^2}{t^2})^{p \ell +1} \,ds \,dt \\
    &= M \int_{t=2}^{\infty} \int_{s= -\frac{1}{t^m}}^{\frac{1}{t^m}} t^{2p \ell + 2} \sum\limits_{j=0}^{p \ell +1} {p \ell +1\choose j}\left(-\frac{s^2}{t^2}\right)^j  \,ds \,dt \\
    &= M \int_{t=2}^{\infty} \left( \sum\limits_{j=0}^{p \ell +1} (-1)^j {p \ell +1\choose j} \frac{s^{2j +1}}{2j+1} t^{2p \ell +2 -2j} \right) \Bigg\vert_{-\frac{1}{t^m}}^{\frac{1}{t^m}} \,dt \\
    &= M_2 \left(\sum_{j=0}^{p \ell +1} (-1)^j {p \ell + 1\choose j} t^{2(p \ell -(m+1)j)+3-m}\right)\Bigg\vert_{t=2}^{\infty} < \infty.
\end{align*}
Thus, for all $0 \leq j \leq p \ell +1$, we must have $2(p \ell -(m+1)j)+3-m < 0$. In particular, $\ell < \frac{m-3}{2p}$ when $j=0$. Setting $m= 2p(k+1) + 3$, we obtain that all monomials $z^{\ell}_1 \dots z^{\ell}_n$ with $0 \leq \ell \leq k$ lie in $A^p(\cup_{1}^n X_i \cup B_{ 2p(k+1) + 3})$. 

We include a poldisk to the above Reinhardt domain to make it a complete Reinhardt domain. More accurately, define $\Omega_k = \cup_{1}^n X_i \cup B_{ 2p(k+1) + 3} \cup \Delta_{2e}$, where $\Delta_{2e}$ is a polydisk of polyradius $2e$. Evidently, $A^p(\Omega_k) = A^p(\cup_{1}^n X_i \cup B_{ 2p(k+1) + 3})=span\set{z_1^lz_2^l\cdots z_n^l: 0\leq l\leq k}$.

\section{Appendix}
This section concerns proving the following proposition which is used in characterizing $p$-Bergman spaces of Hartogs domains with one dimensional base.
\begin{Proposition}\label{prop:key1dimJucha}
Let $\alpha_1, \alpha_2, \alpha_3 \geq 0$ be such that
\begin{align*}
     \text{either }\normalfont \text{frac}(\alpha_1), \text{frac}(\alpha_2), \text{frac}(\alpha_1 + \alpha_2) &>0,\text{ and }\alpha_3 =0, \\
     \text{or, }\alpha_1 = \alpha_2 = \alpha_3 &= \frac{1}{2}.
\end{align*}
Then for each $1 \leq p < 2$ there exists infinitely many $k \in \mathbb{N}$ such that
\begin{itemize}
    \item $ \normalfont \displaystyle{\sum_1^3 3\text{frac}((2N+pk)\frac{1}{2}) > 1}$.
    \item $ \normalfont \displaystyle{\sum_1^3 \text{frac}((2N+pk)\alpha_j) > 1}$, whenever $1, p\alpha_1, p\alpha_2$ are $\mathbb{Q}$-linearly independent in the first case.
\end{itemize}

\end{Proposition}

\begin{proof}
We consider the case of $\alpha_1 = \alpha_2 = \alpha_3 = \frac{1}{2}$ first. 
Now the condition $\displaystyle{\sum_1^3 \text{frac}((2N+pk)\alpha_j) > 1}$ reduces to finding infinitely many $k\in \N$ such that $\text{frac}((2N+pk)\frac{1}{2})=\text{frac}(pk/2)> 1/3$.\\ 
First, we note that $1\leq p<2$, i.e. $0<3-3p/2\leq 3/2$. Let $L$ be smallest natural number  such that $1< L(3-3p/2)<2$. (This can always be done unless $p=4/3$, in which case any natural number $k$ in $3\N+2$ works.)\\ 
Let $k\in \N$ be any natural number. If $pk/2\in (M_1+1/3, M_1+1)$ for some $M_1\in \N$. Then we are done.\\
Else, $M_1\leq pk/2\leq M_1+1/3$ i.e. $3M_1\leq 3pk/2 \leq 3M_1+1$. Then, $ 3M_1-2< 3kp/2-3L+3pL/2< 3M_1$. Therefore $3(M_1+L-1)+1<3(k+L)p/2<3(M_1+L-1)+3$ i.e. if the chosen $k\in \N$ does not meet the condition, $k+L$ satisfies the condition proving that we can get such infinitely many natural numbers.

A natural context for the remaining case is perhaps to consider the sequence $\{(2N\alpha_1+pk\alpha_1, 2N\alpha_2 + pk\alpha_2):k \in \mathbb{N}\}$ as a forward orbit 
$$
{Orb}_+((2N\alpha_1, 2N\alpha_2)):= \{T^k((2N\alpha_1, 2N\alpha_2)): k \in \mathbb{N}\},
$$ where $T: \mathbb{R}^2 \to \mathbb{T}^2$ is a torus action given by $T: (x_1, x_2) \mapsto (x_1 + p \alpha_1, x_2 + p\alpha_2)$. Using theorem $2.7$ of (\cite{rezakhan}, page $17$) or corollary $4.15$ of \cite{manfred259ergodic}, if $1, p\alpha_1, p\alpha_2$ are $\mathbb{Q}$-linearly independent then ${Orb}_+((2N\alpha_1, 2N\alpha_2))$ is dense in $\mathbb{T}^2$. In particular, $\text{frac}(2N\alpha_1 + pk\alpha_1), \text{frac}(2N\alpha_2 + pk\alpha_2)>\frac{1}{2}$ for infinitely many $k \in \mathbb{N}$. 
\end{proof}

The requirement of $p\alpha_1, p\alpha_2$ being $\mathbb{Q}$-linearly independent is necessary. We illustrate it by the example below.
\begin{example}\label{ex:Qlindep}
Set, $p = \frac{3}{2}, \alpha_1 = \frac{1}{12}, \alpha_2 = \frac{7}{12}, N = 3$. Since $p\alpha_1, p\alpha_2$ is rational, the slope of the straight line $\{(tp\alpha_1, t p \alpha_2): t \in \mathbb{R}\}$ is rational and therefore its image inside the torus is periodic. In particular, the forward orbit $\text{Orb}_+((2N\alpha_1, 2N\alpha_2)):= \{T^k((2N\alpha_1, 2N\alpha_2)): k \in \mathbb{N}\}$ is actually a finite set that repeatedly occurs with each period. We can see that 
\begin{align*}
\text{frac}(2N\alpha_1 + jp\alpha_1)+\text{frac}(2N\alpha_2+ jp\alpha_2) &= \begin{cases} 
0 & \text{if } j\equiv 4(mod 8)\\ 
1& \text{else.} \end{cases}
\end{align*}
\end{example}
Thus for the above choice of $p, \alpha_j, N$ the terms $\text{frac}(2N\alpha_1 + kp\alpha_1) + \text{frac}(2N\alpha_2 + kp\alpha_2) \leq 1$, for all $k \in \mathbb{N}$. The required analog of lemma $4.4$ of \cite{jucha2012remark} in the $p$-Bergman space setting is not present. This forces us to restrict the scope of theorem \ref{th:suffpJucha} and proposition \ref{prop:suffpJucha1} to the case $\mathbb{Q}$-linearly independent weights.

It is however important to point out that the above counter-example do not address the main issue. Lemma $4.4$ in \cite{jucha2012remark} establishes $\sum_1^2 \text{frac}(ka_j) > 1$ for infinitely many $k \in \mathbb{N}$ whenever $\text{frac}(a_j), \text{frac}(a_1 + a_2) > 0$. Proposition \ref{prop:key1dimJucha} is therefore an investigation to check if the same forward orbit appears infinitely many times on the sector $\{x+y>1\} \subset \mathbb{T}^2$ of the torus (identified with the square $(x,y) \in [0,1] \times [0,1]$ with usual identification) when the orbit is translated by $\text{frac}(2Na_1) + \text{frac}(2Na_2)$. In the example \ref{ex:Qlindep} however $\text{frac}(pa_1 + pa_2) =0$ and thus $\sum_1^2 \text{frac}(k pa_j) \leq 1$. It is not clear if the translation $(2Na_1 + kpa_1, 2Na_2 + kpa_2) \text{mod} 1$ of the forward orbit $(ka_1, ka_2) \text{mod} 1$ appears infinitely often in the sector $\{x+y>1\}$ inside $\mathbb{T}^2$.

\subsection*{Acknowledgements}
The authors express their gratitude to Harold Boas for suggesting the problem to the first author and for offering many valuable suggestions that greatly enhanced the readability of the paper. The first author thanks Sivaguru for fruitful meetings. They are grateful to Peter Ebenfelt and Ming Xiao for many discussions on the topic and for several important suggestions. The second author thanks Soumya Ganguly and Pranav Upadrashta for many helpful conversations. The second author thanks Ji\v{r}\'{\i} Lebl, Roland Roeder and Sean Curry for their inspiration and encouragement.

\bibliographystyle{alpha}  
\bibliography{references}

\fontsize{11}{9}\selectfont

\vspace{0.5cm}

\noindent TIFR Centre for Applicable Mathematics, Bengaluru, India 560065.

\vspace{0.2cm}

\noindent Email: shreedhar24@tifrbng.res.in

\vspace{0.5 cm}

\noindent Department of Mathematics, University of California San Diego, USA 92093.

\vspace{0.2 cm}

\noindent Email: anandi@ucsd.edu

\end{document}